\documentclass{aims}
\usepackage{amsmath}
  \usepackage{paralist}
  \usepackage{graphics} 
  \usepackage{epsfig} 
 \usepackage[colorlinks=true]{hyperref}

\allowdisplaybreaks[4]

\usepackage{graphicx,amssymb,mathrsfs,latexsym,amsthm}
\usepackage{verbatim}

\hypersetup{urlcolor=blue, citecolor=red}

\def\currentvolume{X}
 \def\currentissue{X}
  \def\currentyear{200X}
   \def\currentmonth{XX}
    \def\ppages{X--XX}

\newtheorem{theorem}{Theorem}[section]
\newtheorem{corollary}{Corollary}

\newtheorem{lemma}[theorem]{Lemma}
\newtheorem{proposition}{Proposition}

\theoremstyle{definition}
\newtheorem{definition}[theorem]{Definition}
\newtheorem{remark}{Remark}

\title[Topological Pressures for {\large$\epsilon$}-Stable  and  Stable Sets]
      {Topological Pressures for {\large$\epsilon$}-Stable  and  Stable Sets }

\author[Xianfeng Ma, Ercai Chen]{}


 \email{xianfengma@gmail.com}
 \email{ecchen@njnu.edu.cn}

\thanks{The first author is supported by the Fundamental Research Funds for
the Central Universities. The second and  first authors are  supported by the National Natural
Science Foundation of China (Grant No. 10971100).
 The second author is partially supported by National
  Basic Research Program of China (973 Program) (Grant No. 2007CB814800)}

\begin{document}
\maketitle

\centerline{\scshape Xianfeng Ma }
\medskip
{\footnotesize
   \centerline{Department of Mathematics}
   \centerline{ East China University of Science and Technology, Shanghai 200237, China}
} 

\medskip

\centerline{\scshape Ercai Chen}
\medskip
{\footnotesize
 \centerline{School of Mathematical Science}
   \centerline{Nanjing Normal University, Nanjing 210097, China}
   \centerline{and}
   \centerline{Center of Nonlinear Science}
   \centerline{Nanjing University, Nanjing 210093, China}
} %

\bigskip


\begin{abstract}
Topological pressures of the preimages of $\epsilon$-stable sets and some certain closed subsets of stable sets in positive entropy systems are investigated. It is showed that the topological pressure of any topological system can be calculated in terms of the topological pressure of the preimages of $\epsilon$-stable sets. For the constructed closed subset of the stable set or the unstable set of any point in a measure-theoretic `rather big' set of  a topological system with positive entropy, especially for the weakly mixing subset contained in the closure of  stable set and unstable set, it is proved that topological pressures of these subsets can be  no less than the measure-theoretic pressure.

\end{abstract}

\section{Introduction}\label{sec1}

Throughout this paper, by a \emph{topological dynamical system} $(X,T)$ (TDS for short) we mean a compact metric space $X$ with a homeomorphism  $T$ from $X$ to itself; the metric on $X$ is denoted by $d$. For $x\in X$ and $\epsilon >0$, the $\epsilon$-\emph{stable } set of $x$ under $T$ is the set of points whose forward orbit $\epsilon$-shadows that of $x$:
\begin{equation*}
W^s_{\epsilon}(x,T)=\{y\in X: d(T^nx,T^ny)\leq \epsilon \,\, \text{for all}\,\, n=1,2,\cdots\}.
\end{equation*}
The preimages of these sets can be nontrivial and hence disperse at a nonzero exponent rate. The dispersal rate function $h_{s}(T,x,\epsilon)$ was introduced in \cite{Fiebig2003}. The relationship between $h_{s}(T,x,\epsilon)$ and the topological entropy $h_{\text{top}}(T)$ was also investigated. It was proved that when $X$ has finite covering dimension, then for all $\epsilon>0$,
$$
\sup_{x\in X}h_{s}(T,x,\epsilon)=h_{\text{top}}(T).
$$
In \cite{Huang2008CMP}, the finite-dimensionality hypothesis turns out to be redundant. This equality is proved to be always true for any non-invertible TDS.

Given $f\in C(X,\mathbb{R})$, consider the \emph{topological pressure of the preimages of }$\epsilon$-\emph{stable set of }$x$:
$$
P(T,f,x,\epsilon)=\lim_{\delta\rightarrow 0}\limsup_{n\rightarrow +\infty}\frac{1}{n}\log P_n(T,f,\delta, T^{-n}W^s_{\epsilon}(x,T)),
$$
where
\begin{align*}
P_n(T,f,\delta, T^{-n}W^s_{\epsilon}(x,T))=\sup\{\sum_{x\in E}\exp f_n(x): E\,\,\text{is an }&(n,\delta) \text{-separated }\\
&\text{subset of } \,\,T^{-n}W^s_{\epsilon}(x,T) \},
\end{align*}
and $f_n(x)=\sum_{i=0}^{n-1}f\circ T^i(x)$.  We show that the topological pressure of any non-invertible TDS with positive metric entropy can be calculated in terms of the topological pressure of the preimages of $\epsilon$-stable sets. That is, for all $\epsilon>0$,
$$
\sup_{x\in X }P(T,f,x,\epsilon)=P(T,f),
$$
where $P(T,f)$ is the standard notion of the topological pressure. Note that for the null function $f$, this is the above equality about the entropy.

For $x\in X$, the \emph{stable set} $W^s(x,T)$ and the \emph{unstable set} $W^u(x,T)$ of $x$ are defined as
\begin{align*}
 &W^s(x,T)=\{ y\in X: \lim_{n\rightarrow +\infty}d(T^nx,T^ny)=0  \},\\
&W^u(x,T)=\{ y\in X: \lim_{n\rightarrow +\infty}d(T^{-n}x,T^{-n}y)=0  \}.
\end{align*}

For Anosov diffeomorphisms on a compact manifold, pairs belonging to the stable set are asymptotic under $T$ and tend to diverge under $T^{-1}$.
However, Blanchard \emph{et al.} \cite{Blanchard2002} showed that in most case, this phenomenon does not happen in a TDS with positive metric entropy.
N. Sumi \cite{Sumi2003} investigated the stable and unstable sets of a $C^2$ diffeomorphism of $C^{\infty}$ manifold with positive metric entropy. He showed that the closure of the stable set $W^s(x,T)$ of `many points' is a perfect $*$-chaotic set and the closure of the unstable set $W^u(x,T)$ contains a perfect $*$-chaotic set.
W. Huang \cite{Huang2008CMP} obtain a stronger result in the general non-invertible TDS with positive metric entropy. He proved that there exists a measure-theoretically `rather big' set such that the closure of the stable or unstable sets of points in the set contains a weakly mixing set. The Bowen entropies of these sets were also estimated. The lower bounded is the usual metric entropy $h_{\mu}(T)$ for the ergodic invariant measure $\mu$.

By introducing the topological pressure for the closed subset and using the excellent partition formed in Lemma 4 of \cite{Blanchard2002},  we show that for the constructed closed subsets of stable and unstable sets in \cite{Huang2008CMP}, the topological pressure of these sets can also be estimated. More precisely, we prove that if $\mu$ is an ergodic invariant measure of a TDS $(X,T)$ with $h_{\mu}(T)>0$, then for $\mu$-a.e. $x\in X$, the constructed closed subsets  $A(x)\subseteq W^s(x,T)$, $B(x)\subseteq W^u(x,T)$ and the weakly mixing subset $E(x)\subseteq \overline{W^s(x,T)}\cap \overline{W^u(x,T)}$ in \cite{Huang2008CMP} satisfies that
\renewcommand{\theenumi}{(\alph{enumi})}
\begin{enumerate}
\item $\lim_{n\rightarrow +\infty} \text{diam}(T^nA(x))=0 \,\, and \,\, P(T^{-1},f,A(x))\geq P_{\mu}(T,f)$;\label{01}
\item $\lim_{n\rightarrow +\infty} \text{diam}(T^{-n}B(x))=0 \,\, and \,\, P(T,f,B(x))\geq P_{\mu}(T,f)$;
\item $P(T,f,E(x))\geq P_{\mu}(T,f)$ and $P(T^{-1},f,E(x))\geq P_{\mu}(T,f)$,\label{02}
\end{enumerate}
where $P_{\mu}(T,f)$ is the measure theoretic pressure.

The paper is organized as follows. In Sec. \ref{sec2}, the topological pressure for the closed subset is introduced. Some related notions and results about entropy are also listed. In Sec. \ref{sec3}, we introduce the notion of topological pressure of the preimages  of $\epsilon$-stable set. Using the tool formed in \cite{Blanchard2002}, we show that the topological pressure of any TDS can be calculated  in terms of our introduced topological pressure of the preimages  of $\epsilon$-stable set. As a generalization of entropy point, the notion of the pressure point is also introduced there. In Sec. \ref{sec4}, we prove the above results \ref{01}-\ref{02}. In Sec. \ref{sec5}, we state and prove the former results for the non-invertible case.

\section{Preliminary}\label{sec2}
Let $(X,T)$ be a TDS and $\mathcal{B}_X$ be the $\sigma$-algebra of
all Borel subsets of $X$. Recall that a {\it cover} of $X$ is a
finite family of Borel subsets of $X$ whose union is $X$, and, a
{\it partition} of $X$ is a cover of $X$ whose elements are pairwise
disjoint. We denote the set of covers, partitions, and open covers,
of $X$, respectively, by $\mathcal{C}_X$, $\mathcal{P}_X$,
$\mathcal{C}_X^o$, respectively.
Given a partition $\alpha$ of $X$ and $x\in X$,
denote $\alpha(x)$ the atom of $\alpha$ containing $x$.
For given two covers $\mathcal{U}$,
$\mathcal{V}\in \mathcal{C}_X$, $\mathcal{U}$ is said to be {\it
finer} than $\mathcal{V}$ (denote by
$\mathcal{U}\succeq\mathcal{V}$) if each element of $\mathcal{U}$ is
contained in some element of $\mathcal{V}$. Let
$\mathcal{U}\vee\mathcal{V}=\{U\cap V: U\in \mathcal{U}, V\in
\mathcal{V}\}$. Given integers $M, N$ with $0\leq M\leq N$ and
$\mathcal{U}\in \mathcal{C}_X$, we denote
$\bigvee_{n=M}^NT^{-n}\mathcal{U}$ by $\mathcal{U}_M^N$.

Given
$\mathcal{U}\in \mathcal{C}_X$ and $K\subset X$, put
$N(\mathcal{U}, K)=\min \{{\rm the\,\, cardinality\,\,
of}\,\,\mathcal{F}: \mathcal{F}\subset\mathcal{U}, \bigcup_{F\in
\mathcal{F}}F\supset K\}$ and $H(\mathcal{U}, K)=\log N(\mathcal{U}, K)$. Then the topological entropy of $\mathcal{U}$ with respect to $T$ for the compact subset $K$ is
\begin{equation*}
 h_{\text{top}}(T,\mathcal{U}, K)=\lim_{n\rightarrow
\infty}\frac{1}{n}H(\mathcal{U}_0^{n-1}, K)=\inf_{n\geq
1}\frac{1}{n}H(\mathcal{U}_0^{n-1}, K).
 \end{equation*}
The topological entropy of $T$ for the compact subset $K$ is defined by $h_{\text{top}}(T,K)=\sup_{\mathcal{U}\in \mathcal{C}^o_X}h_{\text{top}}(T,\mathcal{U}, K)$; and the topological entropy of $T$ is defined by $h_{\text{top}}(T)=\sup_Kh_{\text{top}}(T,K)$.

Let $(X,T)$ be a TDS, $K$ be a closed subset of $X$, $\mathcal{U}\in \mathcal{C}_X^o$ and $f\in C(X,\mathbb{R})$, where $C(X,\mathbb{R})$ be the Banach space of all continuous,
real-valued functions on $X$ endowed with the supremum norm. We denote
\begin{equation}
P_n(T,f,\mathcal{U},K)=\inf\{\sum_{V\in \mathcal{V}}\sup_{x\in V\cap K}\exp f_n(x): \mathcal{V}\in \mathcal{C}_X \,\, \text{and}\,\, \mathcal{V}\succeq \mathcal{U}_0^{n-1}\},
\end{equation}
where $f_n(x)=\sum_{j=0}^{n-1}f(T^jx)$. For $V
\cap K=\emptyset$, we let $f_n(x)=-\infty$ for each $n$.
Then the above definition is well defined. It is clear that if $f$
is the null function, then
$P_n(T,0,\mathcal{U},K)=N(\mathcal{U}_0^{n-1},K)$.

For $\mathcal{V}\in \mathcal{C}_X$, we let $\alpha$ be the Borel
partition generated by $\mathcal{V}$ and denote
\begin{equation*}
\mathcal{P}^*(\mathcal{V})=\{\beta\in \mathcal{P}_X: \beta\succeq
\mathcal{V} \,\, {\rm and\,\, each \,\, atom \,\, of }\,\,\beta \,\,
{\rm is \,\, the\,\, union\,\, of \,\, some\,\, atoms\,\, of \,\,
}\alpha \}.
\end{equation*}

\begin{lemma}[Lemma 2.1 \cite{Xianfeng2009}] \label{sc2.1}
Let $M$ be a compact subset of $X$, $f\in C(X,\mathbb{R})$ and
$\mathcal{V}\in \mathcal{C}_X$. Then
$$
\inf_{\beta\in \mathcal{C}_X, \beta\succeq \mathcal{V}}\sum_{B\in
\beta}\sup_{x\in B\cap M}f(x)=\min\{\sum_{B\in \beta}\sup_{x\in
B\cap M}f(x): \beta\in \mathcal{P}^*(\mathcal{V}) \}.
$$
\end{lemma}

A real-valued function $f$ defined on a compact metric space $Z$ is
called {\it upper semi-continuous }(for short u.s.c.) if one of the
following equivalent conditions holds:
\begin{enumerate}
\item $\limsup_{z'\rightarrow z}f(z')\leq f(z)$ for each $z\in Z$;
\item for each $r\in \mathbb{R}$, the set $\{z\in Z:f(z)\geq r\}$ is
closed.\label{usc2}
\end{enumerate}
By \ref{usc2}, the infimum of any family of u.s.c. functions is
again a u.s.c. one; both the sum and supremum of finitely many
u.s.c. functions are u.s.c. ones.

\begin{lemma}\label{lem3.4}
Let $(X,T)$ be a TDS, $G:x\rightarrow \mathcal{K}(X)$ be a upper semi-continuous  set-valued  mapping, where $\mathcal{K}(X)$ is the collection of all nonempty closed subset of $X$ endowed with the Hausdorff metric, $\mathcal{U}\in\mathcal{C}^o_X$ and $f\in C(X, \mathbb{R})$. Then $F:x\rightarrow \inf\{\sum_{V\in \mathcal{V}}\sup_{y\in V\cap G(x)}f(y):\mathcal{V}\in \mathcal{C}_X \,\, \text{and}\,\, \mathcal{V}\succeq \mathcal{U} \}$ is a u.s.c. function.
\end{lemma}

\begin{proof}
For each $V\in\mathcal{V}\in\mathcal{C}_X$, let $F_V:x\rightarrow \sup_{y\in V\cap G(x)}f(y)$.
For each $r\in \mathbb{R}$, let $\{x_n:n\in \mathbb{N}\}\subseteq \{x:\sup_{y\in V\cap G(x)}f(y)\geq r\}$ be the sequence of points which convergence to the point $x_0$, we show that $x_0\in \{x:\sup_{y\in V\cap G(x)}f(y)\geq r\}$.

Given $\epsilon >0$, since $f\in C(X, \mathbb{R})$, there exists $\delta >0$ such that
\begin{equation}\label{eq3.1}
d(x_1,x_2)<\delta\Rightarrow \mid f(x_1)-f(x_2)  \mid  <\epsilon.
\end{equation}
The set $\{K\in \mathcal{K}(X): H_d(G(x_0),K)<\delta\}$ is an open neighborhood of $G(x_0)$ in $\mathcal{K}(X)$, where $H_d(K_1,K_2)=\inf\{\delta>0: \text{ for all }\,\, x_1\in K_1, x_2\in K_2, d(x_1,K_2)<\delta, d(x_2,K_1)<\delta\}$ is the Hausdorff metric. Since $G$ is a upper semi-continuous set-valued mapping, then for all large enough $n$, $H_d(G(x_n),G(x_0))<\delta$. From \eqref{eq3.1}, we get
\begin{equation*}
\big\vert \sup_{y\in V\cap G(x_n)}f(y)- \sup_{y\in V\cap G(x_0)}f(y) \big\vert \leq \epsilon.
\end{equation*}
By the arbitrary of $\epsilon$, $\sup_{y\in V\cap G(x_0)}f(y)\geq r$. Then $x_0\in \{x:\sup_{y\in V\cap G(x)}f(y)\geq r\}$. That is, $F_V$ is a upper semi-continuous function.

Since the sum of finitely many u.s.c. functions and the infimum of any family of u.s.c. functions are both  u.s.c. functions, we get
\begin{equation*}
F:x\rightarrow \inf\{\sum_{V\in \mathcal{V}}\sup_{y\in V\cap G(x)}f(y):\mathcal{V}\in \mathcal{C}_X \,\, \text{and}\,\, \mathcal{V}\succeq \mathcal{U} \}
\end{equation*}
is a u.s.c. function.
\end{proof}

With a similar argument of Lemma \ref{lem3.4}, we also have the following result.

\begin{lemma}\label{lem3.41}
Let $(X,T)$ be a TDS,   $\mathcal{U}\in\mathcal{C}^o_X$ and $f\in C(X, \mathbb{R})$. Then the function $F: K\rightarrow \inf\{\sum_{V\in \mathcal{V}}\sup_{y\in V\cap K}f(y):\mathcal{V}\in \mathcal{C}_X \,\, \text{and}\,\, \mathcal{V}\succeq \mathcal{U} \}$ is  u.s.c. from $\mathcal{K}(X)$ to $\mathbb{R}$, where $\mathcal{K}(X)$ is the collection of all nonempty closed subset of $X$ endowed with the Hausdorff metric.
\end{lemma}

\begin{proof}
It suffice to prove that for each $V\in\mathcal{V}\in\mathcal{C}_X$,  $F_V:K\rightarrow \sup_{y\in V\cap K}f(y)$ is u.s.c..

For each $r\in \mathbb{R}$, let $\{K_n:n\in \mathbb{N}\}\subseteq \{K:\sup_{y\in V\cap K}f(y)\geq r\}$ be the sequence of subsets which convergence to the set $K_0$, we show that $K_0\in \{K:\sup_{y\in V\cap K}f(y)\geq r\}$.

Given $\epsilon >0$, since $f\in C(X, \mathbb{R})$, there exists $\delta >0$ such that
\begin{equation*}
d(x_1,x_2)<\delta\Rightarrow \mid f(x_1)-f(x_2)  \mid  <\epsilon.
\end{equation*}

For each $K_0\in \mathcal{K}(X)$, the set $\{K\in \mathcal{K}(X): H_d(K,K_0)<\delta\}$ is an open neighborhood of $K$ in $\mathcal{K}(X)$. Then by the definition of the Hausdorff metric, it is easy to know that $\big\vert \sup_{y\in V\cap K}f(y)- \sup_{y\in V\cap K_0}f(y) \big\vert \leq \epsilon.$

By the arbitrary of $\epsilon$, $\sup_{y\in V\cap K_0}f(y)\geq r$. Then $K_0\in \{K:\sup_{y\in V\cap K}f(y)\geq r\}$. That is, $F_V$ is a u.s.c. function.
\end{proof}

\begin{lemma}
Let $(X,T)$ be a TDS, $K$ be a closed subset of $X$, $\mathcal{U}\in\mathcal{C}^o_X$ and $f\in C(X, \mathbb{R})$. Then
\begin{equation*}
P(T,f,\mathcal{U},K)=\lim_{n\rightarrow +\infty}\frac{1}{n}\log P_n(T,f,\mathcal{U},K)
\end{equation*}
exists.
\end{lemma}

\begin{proof}
 For any $n,m\in\mathbb{N}$, $\mathcal{V}_1\succeq
\mathcal{U}_0^{n-1}$, $\mathcal{V}_2\succeq \mathcal{U}_0^{m-1}$, we
have $\mathcal{V}_1\vee T^{-n}\mathcal{V}_2\succeq
\mathcal{U}_0^{n+m-1}$. It follows that
\begin{align*}
P_{n+m}(T,f,\mathcal{U},K)&\leq \sum_{V_1\in
\mathcal{V}_1}\sum_{V_1\in \mathcal{V}_2}\sup_{x\in V_1\cap
T^{-n}V_2\cap K}\exp f_{n+m}(x)\\
&=\sum_{V_1\in \mathcal{V}_1}\sum_{V_2\in \mathcal{V}_2}\sup_{x\in
V_1\cap T^{-n}V_2\cap K}\exp(f_n(x)+f_m(T^nx))\\
&\leq \sum_{V_1\in \mathcal{V}_1}\sum_{V_2\in \mathcal{V}_2}
\big(\sup_{x\in V_1 K}\exp f_n(x)\cdot\sup_{z\in
V_2\cap
T^nK}\exp f_m(z)\big)\\
&=\big(\sum_{V_1\in \mathcal{V}_1}\sup_{x\in V_1\cap
K}\exp f_n(x)  \big)\big(\sum_{V_2\in
\mathcal{V}_2}\sup_{z\in V_2\cap T^nK}\exp f_m(z)  \big).
\end{align*}
Since $\mathcal{V}_i,i=1,2$ is arbitrary, then
$P_{n+m}(T,f,\mathcal{U},K)\leq P_n(T,f,\mathcal{U},K)\cdot P_m(T,f,\mathcal{U},T^nK)$, and so $\log
P_n(T,f,\mathcal{U},K) $ is subadditive. Since $P_n(T,f,\mathcal{U},K)$ is measurable from Lemma \ref{lem3.41}, then by Kingman's subadditive ergodic theorem (See \cite{Walters}), we
complete the proof.
\end{proof}

If $K=X$, then $P(T,f,\mathcal{U},X)=P(T,f,\mathcal{U})$, which is the local topological pressure defined by Huang \emph{et al.} \cite{huang2007}. Clearly, $P(T,0,\mathcal{U},K)=h_{\text{top}}(T,\mathcal{U},K)$.

Let $(X,T)$ be a TDS. Denote by $ \mathcal{M}(X) $ the set of all
Borel, probability measures on $X$, $\mathcal{M}(X,T)$ the set of
$T$-invariant measures, and $\mathcal{M}^e(X,T)$ the set of ergodic
measures. Then
$\mathcal{M}^e(X,T)\subset\mathcal{M}(X,T)\subset\mathcal{M}(X)$,
and $ \mathcal{M}(X), \mathcal{M}(X,T)$ are convex, compact metric
spaces endowed with the weak*-topology.

Given a partition $\alpha\in \mathcal{P}(X)$, $\mu\in
\mathcal{M}(X)$ and a sub-$\sigma$-algebra
$\mathcal{C}\subseteq\mathcal{B}_{\mu}$, let
\begin{equation*}
H_{\mu}(\alpha)=\sum_{A\in \alpha}-\mu(A)\log \mu(A)\,\,\text{and}\,\,
H_{\mu}(\alpha\mid\mathcal{C})=\sum_{A\in\alpha}\int_X-\mathbb{E}(1_A\mid\mathcal{C})\log
\mathbb{E}(1_A\mid\mathcal{C})d\mu,
\end{equation*}
where $\mathbb{E}(1_A\mid\mathcal{C})$ is the expectation of $1_A$
with respect to $\mathcal{C}$. One standard fact states that
$H_{\mu}(\alpha\mid\mathcal{C})$ increases with respect to $\alpha$
and decreases with respect to $\mathcal{C}$. The measure-theoretic entropy of $\mu$ is defined as
$$h_{\mu}(T)=\sup_{\alpha\in \mathcal{P}_X}h_{\mu}(T,\alpha),$$ where $$h_{\mu}(T,\alpha)=\lim_{n\rightarrow +\infty}\frac{1}{n}H_{\mu}(\alpha^{n-1}_0)=\inf_{n\geq 1}H_{\mu}(\alpha^{n-1}_0).$$
For each $f\in C(X,\mathbb{R})$, the \emph{measure-theoretic pressure of} $\mu$ is defined as
 $$P_{\mu}(T,f)=h_{\mu}(T)+\int_Xfd\mu.$$

For a given $\mathcal{U}\in \mathcal{C}_X$, set
\begin{equation*}
H_{\mu}(\mathcal{U})=\inf_{\beta\in \mathcal{P}_X, \,\beta\succeq \mathcal{U}}H_{\mu}(\beta) \,\,\text{and}\,\, H_{\mu}(\mathcal{U}\mid \mathcal{C})=\inf_{\beta\in \mathcal{P}_X, \,\beta\succeq \mathcal{U}}H_{\mu}(\beta\mid \mathcal{C}).
\end{equation*}
When $\mu\in\mathcal{M}(X,T)$ and $\mathcal{C}$ is $T$-invariant (i.e. $T^{-1}\mathcal{C}=\mathcal{C}$),  $H_{\mu}(\mathcal{U}^{n-1}_0\mid \mathcal{C})$ is a non-negative subadditive sequence for a given $\mathcal{U}\in \mathcal{U}$. Let
\begin{equation*}
h_{\mu}(T,\mathcal{U}\mid \mathcal{C})=\lim_{n\rightarrow +\infty}\frac{1}{n}H_{\mu}(\mathcal{U}^{n-1}_0\mid \mathcal{C})=\inf_{n\geq 1}H_{\mu}(\mathcal{U}^{n-1}_0\mid \mathcal{C}).
\end{equation*}
For $\mathcal{C}=\{\emptyset, X\}(\text{mod} \,\,\mu)$, we write $H_{\mu}(\mathcal{U}\mid\mathcal{C})$ and $h_{\mu}(T,\mathcal{U}\mid\mathcal{C})$ by $H_{\mu}(\mathcal{U})$ and $h_{\mu}(T,\mathcal{U})$ respectively. Romagnoli \cite{Rom2003} proved
$$h_{\mu}(T)=\sup_{\mathcal{U}\in \mathcal{C}^o_X}h_{\mu}(T,\mathcal{U}).$$
It is well known that for
$\beta \in\mathcal{P}_X$, $h_{\mu}(T,\beta)=h_{\mu}(T,\beta\mid P_{\mu}(T))\leq H_{\mu}(\beta\mid P_{\mu}(T))$, where $P_{\mu}(T)$ is the Pinsker $\sigma$-algebra of $(X,\mathcal{B}_{\mu},\mu,T)$. Huang \cite{Huang2008CMP} showed the following result.
\begin{lemma}[Lemma 2.1 \cite{Huang2008CMP}]\label{4.2}
Let $(X,T)$ be a TDS, $\mu\in \mathcal{M}(X,T)$ and $\mathcal{U}\in \mathcal{C}_X$. Then
\begin{equation*}
h_{\mu}(T,\mathcal{U})=h_{\mu}(T,\mathcal{U}\mid P_{\mu}(T)).
\end{equation*}
\end{lemma}

For $\mathcal{U}\in \mathcal{C}_X^o$, $\mu\in \mathcal{M}(X,T)$ and  $f\in C(X,\mathbb{R})$,
we define the \emph{measure-theoretic pressure for $T$ with respect to $\mathcal{U}$} as
\begin{equation*}
P_{\mu}(T,f,\mathcal{U})=h_{\mu}(T,\mathcal{U})+\int_Xfd\mu.
\end{equation*}
Obviously,
\begin{equation*}
P_{\mu}(T,f)=h_{\mu}(T)+\int_Xfd\mu=\sup_{\mathcal{U}\in
\mathcal{C}_X^o }h_{\mu}(T,\mathcal{U})+\int_Xfd\mu
=\sup_{\mathcal{U}\in \mathcal{C}_X^o } P_{\mu}(T,f,\mathcal{U}).
\end{equation*}

Let $(X,T)$ be a TDS, $\mu \in \mathcal{M}(X,T)$ and $\mathcal{B}_{\mu}$ be the completion of $
\mathcal{B}_X$ under $\mu$. Then $(X,\mathcal{B}_{\mu},\mu,T)$ be a Lebesgue system. If $\{\alpha_i\}_{i\in I}$ is a countable family of finite partitions of $X$, the partition $\alpha=\bigvee_{i\in I}\alpha_i$ is called a \emph{measurable partition}. The sets $A\in \mathcal{B}_{\mu}$, which are unions of atoms of $\alpha$, form a sub-$\sigma$-algebra of $\mathcal{B}_{\mu}$ by $\widehat{\alpha}$ or $\alpha$ if there is no ambiguity. Every sub-$\sigma$-algebra of $\mathcal{B}_{\mu}$ coincides with a $\sigma$-algebra constructed in this way (mod $\mu$).

Given a measurable partition $\alpha$, put $\alpha^-=\bigvee_{n=1}^{\infty}T^{-n}\alpha$ and $\alpha^T=\bigvee_{n=-\infty}^{+\infty}T^{-n}\alpha$. Define in the same way $\mathcal{F}^-$ and $\mathcal{F}^T$ if $\mathcal{F}$ is a sub-$\sigma$-algebra of $\mathcal{B}_{\mu}$. It is clear that for a measurable partition $\alpha$ of $X$, $\widehat{\alpha^-}=(\widehat{\alpha})^-$ and $\widehat{\alpha^T}=(\widehat{\alpha})^T$ (mod $\mu$).

Let $\mathcal{F}$ be a sub-$\sigma$-algebra of $\mathcal{B}_{\mu}$ and $\alpha$ be the measurable partition of $X$ with $\alpha^-=\mathcal{F}$ (mod $\mu$). $\mu $ can be disintegrated over $\mathcal{F}$ as $\mu=\int_X \mu_x d\mu(x)$, where $\mu_x\in \mathcal{M}(X)$ and $\mu_x(\alpha(x))=1$ for $\mu$-a.e. $x\in X$. The disintegration is characterized by the properties \ref{p1} and \ref{p2} below:
\renewcommand{\theenumi}{(\alph{enumi})}
\begin{enumerate}
\item for every $ f\in L^1(X,\mathcal{B}_X,\mu)$, $f\in L^1(X,\mathcal{B}_X,\mu_x)$ for $\mu$-a.e. $x\in X$,    and the map  \\$x\mapsto \int_Xf(y)d\mu_x(y)$ is in $L^1(X,\mathcal{F},\mu)$;\label{p1}
\item for every $ f\in L^1(X,\mathcal{B}_X,\mu)$, $\mathbb{E}_{\mu}(f\mid \mathcal{F})(x)=\int_Xfd\mu_x$  for $\mu$-a.e. $x\in X$.\label{p2}
\end{enumerate}
Then for any  $f\in L^1(X,\mathcal{B}_X,\mu)$,
$$
\int_X\big(\int_Xfd\mu_x    \big)d\mu(x)=\int_Xfd\mu.
$$

The following lemma was proved in \cite{Huang2008CMP}.

\begin{lemma}[Lemma 2.2 \cite{Huang2008CMP}]\label{lem2.2}
Let $(X,T)$ be a TDS, $\mu\in \mathcal{M}(X,T)$ and $\mathcal{F}$ be a sub-$\sigma$-algebra of $\mathcal{B}_{\mu}$. If $\mu=\int_X\mu_x d\mu(x)$ is the disintegration of $\mu$ over $\mathcal{F}$, then
\renewcommand{\theenumi}{(\alph{enumi})}
\begin{enumerate}
\item for $\mathcal{V}\in \mathcal{C}_X$, $H_{\mu}(\mathcal{V}\mid\mathcal{F})=\int_XH_{\mu_x}(\mathcal{V})d\mu(x)$,
\item for $\mathcal{U}$, $\mathcal{V}\in \mathcal{C}_X$, $H_{\mu}(\mathcal{U}\vee \mathcal{V}\mid \mathcal{F})\leq H_{\mu}(\mathcal{U}\mid \mathcal{F}) +H_{\mu}( \mathcal{V}\mid \mathcal{F})$.
\end{enumerate}
\end{lemma}

Let $K$ be a non-empty closed subset of $X$. For $\epsilon >0$, a subset of $X$ is called an $(n,\epsilon)$-\emph{spanning set} of $K$, if for any $x\in K$ there exists $y\in F$ with $d_n(x,y)\leq \epsilon$, where $d_n(x,y)=\max_{i=0}^{n-1}d(T^ix,T^iy)$; a subset $E$ of $K$ is called an $(n,\epsilon)$-\emph{separated set} of $K$, if $x,y\in E$, $x\neq y$ implies $d_n(x,y)>\epsilon$.
Let $r_n(d,T,\epsilon,K)$ denote the smallest cardinality of any $(n,\epsilon)$-spanning subset for $K$ and $s_n(d,T,\epsilon,K)$ denote the largest cardinality of any $(n,\epsilon)$-separated subset of $K$.

For each $\epsilon>0$ and $f\in C(X,\mathbb{R})$, we denote
\begin{equation*}
P_n(T,f,\epsilon, K)=\sup\{\sum_{x\in E}\exp f_n(x): E \,\,
\text{is  an } \,\, (n,\epsilon)\text{-separated subsets of }\,\, K\}.
\end{equation*}
\emph{The topological pressure of $T$ for the closed subset $K$ }is defined as
\begin{equation*}
P(T,f, K)=\lim_{\epsilon \rightarrow 0}\limsup_{n\rightarrow +\infty}\frac{1}{n}\log P_n(T,f,\epsilon. K).
\end{equation*}
If $f$ is the null function, then $P_n(T,0,\epsilon, K)=s_n(d,T,\epsilon,K)$. It follows that $P(T,f, K)=h(T, K)$, where $h(T,K)$ is the Bowen entropy for the closed subset $K$ defined in \cite{Walters} (see also \cite{Huang2008CMP}).
When $K=X$, $P(T,f,X)=P(T,f)$, where $P(T,f)$ is the standard notion of topological pressure defined in \cite{Walters}. Moreover, it is  not hard to verify that $P(T,f,K)=\sup_{\mathcal{U}\in \mathcal{C}^o_X}P(T,f,\mathcal{U},K)$.

\section{$\epsilon$-stable sets}\label{sec3}

 Let $(X,T)$ be a TDS with a compatible metric $d$. Given $\epsilon >0$, the $\epsilon$-\emph{stable set} of $x$ under $T$ is the set of points whose forward orbit $\epsilon$-shadows that of $x$:
\begin{equation*}
W^s_{\epsilon}(x,T)=\{y\in X: d(T^nx,T^ny)\leq \epsilon \,\, \text{ for all} \,\, n=0, 1, \cdots\}.
\end{equation*}
Since the preimages of these sets can be nontrivial, we can consider the following function. For each $x\in X$, $f\in C(X,\mathbb{R})$ and $\epsilon >0$, let
\begin{equation*}
P_s(T,f,x,\epsilon):=\lim_{\delta\rightarrow 0}\limsup_{n\rightarrow +\infty}\frac{1}{n}\log P_n(T,f,\delta, T^{-n}W^s_{\epsilon}(x,T)).
\end{equation*}
$P_s(T,f,x,\epsilon)$ is called the \emph{topological pressure of the preimages of }$\epsilon$-\emph{stable set of }$x$.
For $f\equiv 0$, $P_s(T,0,x,\epsilon)=h_s(T,x,\epsilon)$, where the latter is the dispersal rate function defined in \cite{Fiebig2003}. It was proved in \cite{Huang2008CMP} that
$\sup_{x\in X}h_s(T,x,\epsilon)=h_{\text{top}}(T)\,\, \text{for all}\,\, \epsilon >0$. In this section, we will show that this is also true for the functions $P_s(T,f,x,\epsilon)$ and $P(T,f)$. By proving that for any $\mu\in \mathcal{M}^e(X,T)$ with positive entropy, $\lim_{\epsilon\rightarrow 0}P_s(T,f,x,\epsilon)\geq P_{\mu}(T,f)$ for $\mu$-a.e. $x\in X$, we get the result. Before proving Theorem \ref{th3.7}, we give several lemmas.

\begin{lemma}\label{lem3.1}
Let $(X,T)$ be a TDS, $f\in C(X,\mathbb{R})$, and $\{K_n\}$ be a sequence of non-empty closed subsets of $X$. Then
\begin{equation*}
\lim_{\delta>0}\limsup_{n\rightarrow +\infty}\frac{1}{n}\log P_n(T,f,\delta,K_n)=\sup_{\mathcal{U}\in \mathcal{C}^0_X}\limsup_{n\rightarrow +\infty}\frac{1}{n}\log P_n(T,f,\mathcal{U}, K_n).
\end{equation*}
\end{lemma}

\begin{proof}
For a fixed $\delta >0$, choose $\mathcal{V}\in \mathcal{C}_X^o$ with $\text{diam}(\mathcal{V})<\delta$. For $n\in \mathbb{N}$ let $A$ be an $(n,\delta)$-separated set of $K_n$. Since $B\cap K_n$ contains at most one element of $A$ for each $B$ of $\bigvee_{i=0}^{n-1}T^{-i}\mathcal{V}$, then for every $\mathcal{W}\in \mathcal{C}_X$ with $\mathcal{W}\succeq \mathcal{V}_0^{n-1}$, each element of $\mathcal{W}$ also contains at most one element of $A$. We get $\sum_{x\in A}\exp f_n(x)\leq P_n(T,f,\mathcal{V},K_n)$. That is $P_n(T,f,\delta, K_n)\leq P_n(T,f,\mathcal{V},K_n)$. Then
\begin{align*}
\limsup_{n\rightarrow +\infty}\frac{1}{n}\log P_n(T,f,\delta, K_n)
&\leq \limsup_{n\rightarrow +\infty}\frac{1}{n}\log P_n(T,f,\mathcal{V},K_n)\\
&\leq \sup_{\mathcal{U}\in \mathcal{C}^o_X}\limsup_{n\rightarrow +\infty}\frac{1}{n}\log P_n(T,f,\mathcal{U},K_n).
\end{align*}
Letting $\delta\rightarrow 0$, we get
\begin{equation*}
\lim_{\delta>0}\limsup_{n\rightarrow +\infty}\frac{1}{n}\log P_n(T,f,\delta,K_n)
\leq
\sup_{\mathcal{U}\in \mathcal{C}^0_X}\limsup_{n\rightarrow +\infty}\frac{1}{n}\log P_n(T,f,\mathcal{U}, K_n).
\end{equation*}

In the following, we show the converse inequality. For any fixed $\mathcal{U}\in \mathcal{C}_X^o$, let $\delta $ be the Lebesgue number of $\mathcal{U}$. For $n\in \mathbb{N}$, let $E$ be an $(n,\frac{\delta}{2})$-separated set of $K_n$ with the largest cardinality. Then $E$ ia also an $(n,\frac{\delta}{2})$-spanning set of $K_n$. From the definition of spanning sets, we know
\begin{equation*}
\bigcup_{x\in E}\bigcap_{i=0}^{n-1}T^{-i}\overline{B_{\frac{\delta}{2}}(T^ix)}\supset K_n,\,\, \text{where}\,\, \overline{B_{\frac{\delta}{2}}(T^ix)}=\{y\in X:d(T^ix,y)\leq \frac{\delta}{2}\}.
\end{equation*}
Now for each $x\in E$ and $0\leq i\leq n-1$, $\overline{B_{\frac{\delta}{2}}(T^ix)}$ is contained in some element of $\mathcal{U}$ since $\delta$ is the Lebesgue number of the open cover $\mathcal{U}$. Hence for each $x\in E$, $\bigcap_{i=0}^{n-1}T^{-i}\overline{B_{\frac{\delta}{2}}(T^ix)}$ is contained in some element of $\bigvee_{i=0}^{n-1}T^{-i}\mathcal{U}$.
Let $\mathcal{W}=\{\bigcap_{i=0}^{n-1}T^{-i}\overline{B_{\frac{\delta}{2}}(T^ix)}:x\in E\}$, then $\mathcal{W}\in \mathcal{C}_X$ and $\mathcal{W}\succeq \mathcal{U}^{n-1}_0$. Let
\begin{equation*}
Q_n(T,f,\mathcal{U},K_n)=\inf\{\sum_{V\in \mathcal{V}}\inf_{x\in V\cap K_n}\exp f_n(x):\mathcal{V}\in \mathcal{C}_X\,\,\text{and }\,\,  \mathcal{V}\succeq \mathcal{U}^{n-1}_0  \},
\end{equation*}
then
\begin{equation*}
Q_n(T,f,\mathcal{U},K_n)\leq \sum_{x\in E}f_n(x)\leq P_n(T,f,\frac{\delta}{2},K_n).
\end{equation*}
Let $\tau_{\mathcal{U}}=\sup\{\mid f(x)-f(y) \mid :d(x,y)\leq \text{diam}(\mathcal{U}) \}$, then
$\exp (-n \tau_{\mathcal{U}})P_n(T,f,\mathcal{U},K_n)\leq Q_n(T,f,\mathcal{U},K_n)$.
So
\begin{align*}
-\tau_{\mathcal{U}}+\limsup_{n\rightarrow +\infty}\frac{1}{n}\log P_n(T,f,\mathcal{U},K_n)
&\leq \limsup_{n\rightarrow +\infty}\frac{1}{n}\log P_n(T,f,\frac{\delta}{2},K_n)\\
& \leq \lim_{\delta\rightarrow 0}\limsup_{n\rightarrow +\infty}\frac{1}{n}\log P_n(T,f,\frac{\delta}{2},K_n).
\end{align*}
Since $\mathcal{U}$ is arbitrary, we get
\begin{align*}
\sup_{\mathcal{U}\in \mathcal{C}^0_X}\limsup_{n\rightarrow +\infty}\frac{1}{n}\log P_n(T,f,\mathcal{U}, K_n)   \leq \lim_{\delta>0}\limsup_{n\rightarrow +\infty}\frac{1}{n}\log P_n(T,f,\delta,K_n).
\end{align*}
\end{proof}

An immediate consequence of Lemma \ref{lem3.1} is the following.

\begin{lemma}\label{lem3.2}
Let $(X,T)$ be a TDS and $f\in C(X,\mathbb{R})$. Then for each $x\in X$ and $\epsilon>0$,
\begin{equation*}
P_s(T,f,x,\epsilon)=\sup_{\mathcal{U}\in \mathcal{C}_X^o}\limsup_{n\rightarrow +\infty} \frac{1}{n} \log P_n(T,f,\mathcal{U},T^{-n}W^s_{\epsilon}(x,T)).
\end{equation*}
\end{lemma}

\begin{lemma}[Lemma 9.9 \cite{Walters}]\label{4.1}
Let $a_1,\cdots,a_k$ be given real numbers. If $p_i\geq0,
i=1,\cdots,k$, and $\sum_{i=1}^kp_i=1$, then
$$
\sum_{i=1}^kp_i(a_i-\log p_i)\leq \log(\sum_{i=1}^k e^{a_i}),
$$
and equality holds iff $p_i=\frac{e^{a_i}}{\sum_{i=1}^k e^{a_i}}$
for all $i=1,\cdots, k$.
\end{lemma}

Let $(X,T)$ be a TDS, $\mu \in \mathcal{M}(X,T)$ and $\mathcal{B}_{\mu}$ be the completion of $\mathcal{B}_X$ under $\mu$. The \emph{Pinsker} $\sigma$\emph{-algebra} $P_{\mu}(T)$ is defined as the smallest sub-$\sigma$-algebra of $\mathcal{B}_{\mu}$ containing $\{\xi\in \mathcal{P}_X:h_{\mu}(T,\xi)=0\}$. It is well known that $P_{\mu}(T)=P_{\mu}(T^{-1})$ and $P_{\mu}(T)$ is $T$-invariant, i.e. $T^{-1}(P_{\mu}(T))=P_{\mu}(T)$. We need the following lemma proved in \cite{Huang2008CMP}.

\begin{lemma}[Lemma 3.5 \cite{Huang2008CMP}]\label{lem3.5}
Let $(X,T)$ be a TDS, $\mu\in \mathcal{M}(X,T)$ and $\delta >0$. Then there exist $\{W_i\}_{i=1}^{\infty}\subset \mathcal{P}_X$ and $0=k_1<k_2<\cdots$ such that
\begin{enumerate}
\item $\text{diam}(W_1)<\delta$ and $\lim\limits_{i\rightarrow +\infty}\text{diam}(W_i)=0$,
\item $\lim\limits_{k\rightarrow +\infty} H_{\mu}(P_k\mid \mathcal{P}^-)=h_{\mu}(T)$, where $P_k=\bigvee_{i=1}^kT^{-k_i}W_i$ and $\mathcal{P}=\bigvee_{k=1}^{\infty}P_k$,
\item $\bigcap_{n=0}^{\infty}\widehat{T^{-n}\mathcal{P}^-}= P_{\mu}(T)$.
\end{enumerate}
\end{lemma}

\begin{lemma}
Let $(X,T)$ be a TDS, $K$ be a closed subset of $X$, $\mathcal{U}\in \mathcal{C}_X^o$ and $f\in C(X,\mathbb(R))$. Then for each $n\in \mathbb{N}$,
\begin{equation*}
P_n(T,f,\mathcal{U},T^{-n}K)=P_n(T,f\circ T^{-n},T^n\mathcal{U},K).
\end{equation*}
\end{lemma}

\begin{proof}
For each $\mathcal{V}\in \mathcal{C}_X$ and $\mathcal{V}\succeq \bigvee_{i=1}^n T^i\mathcal{U}$, obviously, $T^{-n}\mathcal{V}\in \mathcal{C}_x$ and $T^{-n}\mathcal{V}\succeq \bigvee_{i=0}^{n-1}T^{-i}\mathcal{U}$.

Since for each $V\in \mathcal{V}$,
$$\sup_{x\in T^{-n}V\cap T^{-n}K}\exp f_n(x)=\sup_{x\in V\cap K}\exp f_n(T^{-n}x),$$
 it is easy to see that $P_n(T,f,\mathcal{U},T^{-n}K)\leq P_n(T,f\circ T^{-n},T^n\mathcal{U},K)$. From the homeomorphism of $T$, the inverse inequality holds. Then
$P_n(T,f,\mathcal{U},T^{-n}K)=P_n(T,f\circ T^{-n},T^n\mathcal{U},K).$
\end{proof}

Now we proceed to prove the following theorem which clearly implies our main result.

\begin{theorem}\label{th3.7}
Let $(X,T)$ be a TDS, $f\in C(X,\mathbb{R})$ and $\mu \in \mathcal{M}^e(X,T)$ with $h_{\mu}(T)>0$. Then for $\mu$-a.e. $x\in X$, $\lim\limits_{\epsilon\rightarrow 0}P_s(T,f,x,\epsilon)\geq P_{\mu}(T,f)$.
\end{theorem}

\begin{proof}
It suffice to prove that for a given $\epsilon >0$, $P_s(T,f,x,\epsilon)\geq P_{\mu}(T,f)$ for $\mu$-a.e. $x\in X$.

Fix $\epsilon>0$. Since $T$ is homeomorphism on $X$, there exists $\delta \in(0,\epsilon)$ such that $d(T^{-1}x,T^{-1}y)<\epsilon$ when $d(x,y)<\delta$. By Lemma \ref{lem3.5}, there exists $\{P_i\}_{i=1}^{\infty}\subset \mathcal{P}_X$ satisfying that $\text{diam}(P_1)\leq \delta$, $\bigcap_{n=0}^{\infty}\widehat{T^{-n}\mathcal{P}^-}=P_{\mu}(T)$ and $H_{\mu}(P_k\mid \mathcal{P}^-)\rightarrow h_{\mu}(T)$ when $k\rightarrow +\infty$, where $\mathcal{P}=\bigvee_{i=1}^{\infty}P_i$. Since $\text{diam}(P_1)\leq \delta$, it is clear that $\mathcal{P}^-(x)\subseteq W^s_{\epsilon}(x,T) $ for each $x\in X$.

Let $\mu =\int_X\mu_xd\mu(x)$ be the disintegration of $\mu $ over $\mathcal{P}^-$. Then
\begin{equation*}
\text{supp}(\mu_x)\subseteq \overline{\mathcal{P}^-(x)}\subseteq W^s_{\epsilon}(x,T) \,\, \text{for }
\mu \text{-a.e. } \,\,x\in X.
\end{equation*}

Let $k\in \mathbb{N}$. By the inequality of (3.3) in \cite{Huang2008CMP}, we know that  there exists $\mathcal{U}_k\in \mathcal{C}_X^o$ such that
\begin{equation}\label{eq3.3}
\limsup_{n\rightarrow +\infty} \frac{1}{n}H_{\mu}\big( \bigvee_{i=0}^{n-1}T^{-i}\mathcal{U}_k\mid T^{-n}\mathcal{P}^-\big) \geq H_{\mu}(P_k\mid \mathcal{P}^-)-\frac{1}{k}.
\end{equation}

For $n\in \mathbb{N}$, let $F_n(x)=\frac{1}{n}\log P_n(T, f\circ T^{-n}, T^n\mathcal{U}_k, W^s_{\epsilon}(x,T))$. By Lemma \ref{lem3.4}, we get that $F_n$ is u.s.c.. Moreover, it is a Borel measurable bounded function. Let $F(x)=\limsup_{n\rightarrow +\infty} F_n(x)$ for $x\in X$. Then $F$ is Borel measurable map. Since $TW^s_{\epsilon}(x,T)\subseteq W^s_{\epsilon}(Tx,T)$ for each $x\in X$, we have
\begin{align*}
 &P_n(T, f\circ T^{-n}, T^n\mathcal{U}_k, W^s_{\epsilon}(x,T))\\
\leq &\inf\{\sum_{V\in \mathcal{V}}\sup_{y\in V\cap TW^s_{\epsilon}(x,T)}\exp f_n\circ T^{-(n+1)}(x):\mathcal{V}\in \mathcal{C}_X \,\, \text{and }\,\, \mathcal{V}\succeq \bigvee_{i=2}^{n+1}T^i\mathcal{U}_k\}\\
\leq &\inf\{\sum_{V\in \mathcal{V}}\sup_{y\in V\cap W^s_{\epsilon}(Tx,T)}\exp f_n\circ T^{-(n+1)}(x):\mathcal{V}\in \mathcal{C}_X \,\, \text{and }\,\, \mathcal{V}\succeq \bigvee_{i=1}^{n+1}T^i\mathcal{U}_k\}\\
=&P_{n+1}(T,f\circ T^{-(n+1)}, T^{n+1}\mathcal{U}_k, W^s_{\epsilon}(Tx,T)).
\end{align*}
Then
\begin{align*}
 F(x)&=\limsup_{n\rightarrow +\infty}\frac{1}{n}\log P_n(T, f\circ T^{-n}, T^n\mathcal{U}_k, W^s_{\epsilon}(x,T))\\
 &\leq \limsup_{n\rightarrow +\infty}\frac{n+1}{n}\cdot\frac{1}{n+1}\log P_{n+1}(T,f\circ T^{-(n+1)}, T^{n+1}\mathcal{U}_k, W^s_{\epsilon}(Tx,T))\\
 &=F(Tx).
\end{align*}
That is $F(x)\leq F(Tx)$ for each $x\in X$. Since $\mu \in \mathcal{M}(X,T)$, $\int_XF(Tx)d\mu(x)=\int_X F(x)d\mu(x)$, then for $\mu$-a.e. $x\in X$, $F(Tx)=F(x)$. Moreover, $F(x)\equiv a_k$ for $\mu$-a.e. $x\in X$ as $\mu $ is ergodic, where $a_k\geq 0$ is a constant.

From Lemma \ref{sc2.1}, there exists a finite partition $\beta \in \mathcal{P}^*(\bigvee_{i=1}^nT^i\mathcal{U}_k)$ such that
\begin{equation*}
P_n(T, f\circ T^{-n}, T^n\mathcal{U}_k, W^s_{\epsilon}(x,T))
=\sum_{B\in \beta}\sup_{x\in B\cap W^s_{\epsilon}(x,T)}\exp f_n\circ T^{-n}(x).
\end{equation*}
It follows from Lemma \ref{4.1} that
\begin{align*}
&\log P_n(T, f\circ T^{-n}, T^n\mathcal{U}_k, W^s_{\epsilon}(x,T))\\
=&\log \sum_{B\in \beta}\sup_{x\in B\cap W^s_{\epsilon}(x,T)}\exp f_n\circ T^{-n}(x)\\
\geq &\sum_{B\in \beta}\mu_x(B\cap W^s_{\epsilon}(x,T))\big( \sup_{x\in B\cap W^s_{\epsilon}(x,T)}\exp f_n\circ T^{-n}(x)-\log \mu_x(B\cap W^s_{\epsilon}(x,T)) \big)\\
=&H_{\mu_x}(\beta)+\sum_{B\in \beta}\sup_{x\in B\cap W^s_{\epsilon}(x,T)} f_n\circ T^{-n}(x)\cdot \mu_x(B)  \\
 &\quad \quad  \quad \quad  \quad \quad \quad  \quad \quad \quad \quad \,\,(\text{supp}(\mu_x)\subseteq W^s_{\epsilon}(x,T) \,\, \text{for}\,\,\mu\text{-a.e.}\,\, x\in X )                      \\
\geq &H_{\mu_x}(\bigvee_{i=1}^nT^i\mathcal{U}_k)+\int_Xf_n\circ T^{-n}d\mu_x.
\end{align*}
Then
\begin{align*}
&a_k=\int_XF(x)d\mu=\int_X\limsup_{n\rightarrow +\infty}F_n(x)d\mu\geq \limsup_{n\rightarrow +\infty}\int_XF_n(x)d\mu\\
&\geq \limsup_{n\rightarrow +\infty}\int_X\frac{1}{n}\big( H_{\mu_x}(\bigvee_{i=1}^nT^i\mathcal{U}_k)+\int f_n\circ T^{-n}d\mu_x \big)d\mu(x)\\
&=\limsup_{n\rightarrow +\infty}\big( \int_X\frac{1}{n}H_{\mu_x}(\bigvee_{i=1}^nT^i\mathcal{U}_k)d\mu(x) + \frac{1}{n}\int_X  \int f_n\circ T^{-n}d\mu_x  d\mu(x)  \big)\\
&=\limsup_{n\rightarrow +\infty}\big( \int_X \frac{1}{n} H_{\mu_x}(\bigvee_{i=1}^nT^i\mathcal{U}_k)d\mu(x) + \frac{1}{n}\int_X f_n\circ T^{-n} d\mu(x)\\
&=\limsup_{n\rightarrow +\infty}\int_X \frac{1}{n} H_{\mu_x}(\bigvee_{i=1}^nT^i\mathcal{U}_k)d\mu(x)+\int_Xf d\mu(x) \,\,(\text{since}\mu \in \mathcal{M}(X,T)) \\
&=\limsup_{n\rightarrow +\infty}\frac{1}{n} H_{\mu}(\bigvee_{i=1}^nT^i\mathcal{U}_k\mid \mathcal{P}^-)+\int_Xf d\mu(x) \,\,(\text{by Lemma}\,\, \ref{lem2.2}(a))\\
&=\limsup_{n\rightarrow +\infty}\frac{1}{n}H_{\mu}(\bigvee_{i=1}^{n-1}T^{-i}\mathcal{U}_k\mid T^{-n}\mathcal{P}^-)+\int_Xf d\mu(x)\\
&\geq H_{\mu}(P_k\mid \mathcal{P}^-)-\frac{1}{k}+\int_Xf d\mu(x) \,\, (\text{by the inequality }\,\,\eqref{eq3.3}).
\end{align*}
Since $P_s(T,f,x,\epsilon)\geq F(x)$ for each $x\in X$, then
\begin{align*}
P_s(T,f,x,\epsilon)&\geq \lim_{k\rightarrow +\infty}\big( H_{\mu}(P_k\mid \mathcal{P}^-)-\frac{1}{k}+\int_Xf d\mu(x)   \big)\\
&=h_{\mu}(T)+\int_Xf d\mu(x)=P_{\mu}(T,f)
\end{align*}
for $\mu$-a.e. $x\in X$.
\end{proof}

We introduce the $\epsilon$-pressure point and pressure point for a TDS. Let $(X,T)$ be a TDS, $f\in C(X,\mathbb{R})$. For $\epsilon >0$, we call $x\in X$  an $\epsilon$\emph{-pressure point} for $T$ if $P_s(T,f,x,\epsilon)=P(T,f)$; and \emph{pressure point} if $\lim\limits_{\epsilon\rightarrow o}P_s(T,f,x,\epsilon)=P(T,f)$. The function $P_s(T,f,x,\epsilon)$ is decreasing in $\epsilon$. It follows that every pressure point is also an $\epsilon$-pressure point for each $\epsilon>0$. Note that while the notion of $\epsilon$-pressure point depends on the choice of the metric, that of pressure point does not. Denote by $\mathscr{P}(T,f)$ the set of all pressure point of $(X,T)$ for $f\in C(X,\mathbb{R})$. For $f\equiv 0$, the $\epsilon$-pressure point and pressure point are the $\epsilon$-entropy point and  entropy point, respectively, which are introduced in \cite{Fiebig2003}. Moreover, $\mathscr{P}(T,0)=\mathcal{E}(T)$, where $\mathcal{E}$ is the set of all entropy points of $(X,T)$.

\begin{remark}
Let $(X,T)$ be a TDS, $f\in C(X,\mathbb{R})$. If there exists $\mu\in \mathcal{M}^e(X,T)$ such that $P(T,f)=P_{\mu}(T,f)$, then $\mathscr{P}(T,f)\neq \emptyset$.
\end{remark}

\section{Stable sets}\label{sec4}

The main results of this section is Theorem \ref{th4.2} and Theorem \ref{4.6}.
Recall that for a TDS $(X,T)$ and $x\in X$,
\begin{align*}
W^s(x,T)&=\{y\in X: \lim_{n\rightarrow +\infty}d(T^nx,T^ny)=0\} \,\,\text{and}\\
W^u(x,T)&=\{y\in X: \lim_{n\rightarrow +\infty}d(T^{-n}x,T^{-n}y)=0\}
\end{align*}
$W^s(x,T)$ is called the \emph{stable set } of $x$ for $T$, and $W^u(x,T)$ is called the \emph{unstable set} of $x$ for $T$. Obviously, $W^s(x,T)=W^u(x,T^{-1})$ and $W^u(x,T)=W^s(x,T^{-1})$.
To show  Theorem \ref{th4.2}, we need the following lemma.

\begin{lemma}\label{lem4.1}
Let $G:X\rightarrow \mathcal{K}(X)$ be a set-valued measurable map, where $\mathcal{K}(X)$ is the family of all closed subset of $X$ endowed with the Hausdorff metric, $f\in C(X,\mathbb{R})$ and $\mathcal{U}\in \mathcal{C}_X^o$. Then
\begin{equation*}
F: x\rightarrow \inf\{\sum_{V\in \mathcal{V}}\sup_{y\in V\cap G(x)}f(y): \mathcal{V}\in \mathcal{C}_X \,\, \text{and}\,\, \mathcal{V}\succeq \mathcal{U} \}
\end{equation*}
is a Borel measurable map.
\end{lemma}

\begin{proof}
By Lemma \ref{sc2.1}, we have
\begin{align*}
\inf\{\sum_{V\in \mathcal{V}}\sup_{y\in V\cap G(x)}f(y): \mathcal{V}\in \mathcal{C}_X \,\, \text{and}\,\, \mathcal{V}\succeq \mathcal{U} \}
=\min\sum_{V\in \mathcal{V}}\sup_{y\in V\cap G(x)}f(y): \mathcal{V}\in \mathcal{P}^*(\mathcal{U})\}.
\end{align*}
It is sufficient to prove that for each $\mathcal{V}\in \mathcal{P}^*(\mathcal{U})$, the function
 $H: x\rightarrow \sup\limits_{y\in V\cap G(x)}f(y)$
 is measurable .

Since $G:X\rightarrow \mathcal{K}(X)$ is a set-valued measurable map, by the well-known Castaing representation theorem, there exists a countable family $\{\sigma_n:n\in \mathbb{N}\}$ of measurable selections of $G$ such that $G(x)=\overline{\{\sigma_n(x):n\in \mathbb{N}\}}$  for each $x\in X$. Then $G$ admits a subsequence $\{\sigma_{n_i}:i\in \mathbb{N}\}$ such that $\overline{V}\cap G(x)=\overline{\{\sigma_{n_i}(x):i\in \mathbb{N}\}}$. It follows that $H(x)=\sup_{n\geq 1}f(\sigma_n(x))$ is a Borel measurable function. Then $F$ is a Borel measurable map.
\end{proof}

\begin{theorem}\label{th4.2}
Let $(X,T)$ be a TDS, $f\in C(X,\mathbb{R})$ and $\mu\in \mathcal{M}^e(X,T)$ with $h_{\mu}(T)>0$. Then for $\mu$-a.e. $x\in X$,
\begin{enumerate}
\item
there exists a closed subset $A(x)\subseteq W^s(x,T)$ such that
\begin{equation*}
\lim_{n\rightarrow +\infty} \text{diam}(T^nA(x))=0 \,\, and \,\, P(T^{-1},f,A(x))\geq P_{\mu}(T,f); \end{equation*}\label{001}
\item
there exists a closed subset $B(x)\subseteq W^u(x,T)$ such that
\begin{equation*}
\lim_{n\rightarrow +\infty} \text{diam}(T^{-n}B(x))=0 \,\, and \,\, P(T,f,B(x))\geq P_{\mu}(T,f). \end{equation*}\label{002}
\end{enumerate}
\end{theorem}

\begin{proof}
Since $\mu\in \mathcal{M}^e(X,T)$, $P_{\mu}(T^{-1},f)=P_{\mu}(T,f)$ and $W^s(x,T^{-1})=W^u(x,T)$, \ref{001} implies \ref{002}. It remains to prove \ref{001}.

By Lemma \ref{lem3.5}, there exist $\{W_i\}_{i=1}^{\infty}\subset \mathcal{P}_X$ and $0=k_1<k_2<\cdots$ satisfying that
\begin{enumerate}
\item $\text{diam}(W_1)<\delta$ and $\lim\limits_{i\rightarrow +\infty}\text{diam}(W_i)=0$,
\item $\lim\limits_{k\rightarrow +\infty} H_{\mu}(P_k\mid \mathcal{P}^-)=h_{\mu}(T)$, where $P_k=\bigvee_{i=1}^kT^{-k_i}W_i$ and $\mathcal{P}=\bigvee_{k=1}^{\infty}P_k$,
\item $\bigcap_{n=0}^{\infty}\widehat{T^{-n}\mathcal{P}^-}= P_{\mu}(T)$.
\end{enumerate}
Let $Q_i=\bigvee_{j=1}^iT^{-j}(P_1\vee P_2\vee \cdots P_i)$ for $i\in \mathbb{N}$. Then $Q_i\in \mathcal{P}_X$, $Q_1\preceq Q_2\preceq \cdots $ and $\bigvee_{i=1}^{\infty}Q_i=\mathcal{P}^-$.

For $x\in X$, let $A(x)=\bigcap_{i=1}^{\infty}\overline{Q_i(x)}$. Then $A(x)$ is a closed set and $A(x)\supseteq \overline{\mathcal{P}^-(x)}$. The set $A(x)$ also satisfies that $\lim\limits_{n\rightarrow +\infty}\text{diam}(T^nA(x))=0$ and $A(x)\subseteq W^s(x,T)$ (see the proof of Theorem 4.2 \cite{Huang2008CMP} for details).

 Moreover, the set-valued map $A:x\rightarrow A(x)$ is measurable. In fact, for each open set $U$ of $X$,
\begin{equation*}
\{x:\bigcap_{n=1}^{\infty}\overline{Q_i(x)}\subseteq U\}
=\bigcup_{n\geq 1}\bigcap_{k\geq n}\bigcap\{ A\in Q_k: \overline{A}\subseteq U\}
\end{equation*}
 is a Borel set of $X$. Then for  each closed set $V$ of $X$, $\{x:\overline{Q_i(x)}\subseteq X \backslash V\}$ is a Borel set. It follows that $\{x:\overline{Q_i(x)}\cap V \neq \emptyset\}$ is  Borel  and then $A:x\rightarrow A(x)$ is set-valued measurable.

Let $\mu=\int_X \mu_x d\mu(x)$ be the disintegration of $\mu$ over $\mathcal{P}^-$. Then
\begin{equation}\label{eq4.2}
\text{supp}(\mu_x)\subseteq \overline{\mathcal{P}^-(x)}\subseteq A(x) \,\, \text{for}\,\, \mu \text{-a.e.}\,\, x\in X.
\end{equation}

We now prove that for $\mu$-a.e. $x\in X$, $P(T^{-1},f, A(x))\geq P_{\mu}(T,f)$. Since $\lim\limits_{k\rightarrow +\infty}H_{\mu}(P_k\mid \mathcal{P}^-)=h_{\mu}(T)$, it is sufficient to prove that for each $k\in \mathbb{N}$, $P(T^{-1},f, A(x))\geq H_{\mu}(P_k\mid \mathcal{P}^-)-\frac{1}{k}+\int_X f d\mu(x)$ for $\mu$-a.e. $x\in X$.

For a given $k\in \mathbb{N}$, there exists $\mathcal{U}_k\in \mathcal{C}_X^o$ such that
\begin{equation}\label{eq4.3}
\limsup_{n\rightarrow +\infty}\frac{1}{n}H_{\mu}\big(\bigvee_{i=0}^{n-1}T^{-i}\mathcal{U}_k\mid T^{-n}\mathcal{P}^- \big) \geq H_{\mu}(P_k\mid \mathcal{P}^-)-\frac{1}{k} \,\, \text{for each}\,\, n\in\mathbb{N}
\end{equation}
(see \cite{Huang2008CMP} for details).

Let $F_n(x)=\frac{1}{n}\log P_n(T^{-1},f,\mathcal{U}_k,A(x)) $, where
\begin{align*}
P_n(T^{-1},f,\mathcal{U}_k,A(x))=\inf\{ \sum_{V\in\mathcal{V}}\sup_{y\in V\cap A(x)}\exp &f_n\circ T^{_(n-1)}(y):\\
&\mathcal{V}\in \mathcal{C}_X \,\, \text{and}\,\, \mathcal{V}\succeq \bigvee_{i=0}^{n-1}T^i\mathcal{U}_k \}.
\end{align*}
By the above Lemma \ref{lem4.1}, $F_n$ is a Borel measurable function. Let $F_n(x)=\limsup\limits_{n\rightarrow +\infty}F_n(x)$ for each $x\in X$. Then $F$ is also a Borel measurable function on $X$.

For each $\mathcal{V}\succeq \bigvee_{i=0}^{n-1}T^i\mathcal{U}_k$, $T^{-1}\mathcal{V}\succeq \bigvee_{i=0}^{n-1}T^i\mathcal{U}_k$. Since $T(A(x))\subseteq A(T(x))$ (see the proof of Theorem 4.2 \cite{Huang2008CMP}), for each $V\in \mathcal{V}$,
\begin{align*}
\sup_{y\in T^{-1}V\cap A(x)}\sum_{i=0}^{n-1}&f(T^{-i}y)\leq \sup_{y\in T^{-1}(V\cap A(Tx))}\sum_{i=0}^{n-1}f(T^{-i}y)\\
&=\sup_{y\in V\cap A(Tx)}\sum_{i=1}^n f(T^{-i}y)\leq \sup_{y\in V\cap A(Tx)}\sum_{i=0}^n f(T^{-i}y),
\end{align*}
then it is not hard to see that
\begin{equation*}
P_n(T^{-1},f,\mathcal{U}_k, A(x))\leq P_{n+1}(T^{-1},f,\mathcal{U}_k, A(Tx)).
\end{equation*}
It follows that
\begin{align*}
F(x)&=\limsup_{n\rightarrow +\infty}\frac{1}{n}\log P_n(T^{-1},f,\mathcal{U}_k, A(x))\\
&\leq \limsup_{n\rightarrow +\infty}\frac{n+1}{n}\cdot\frac{1}{n+1}\log P_n(T^{-1},f,\mathcal{U}_k, A(Tx))\\
&= F(Tx).
\end{align*}
That is $F(x)\leq F(Tx)$ for each $x\in X$. Since $\mu \in \mathcal{M}(X,T)$, $\int_X (f(Tx)-f(x))d\mu(x)=0$, then $F(Tx)=F(x)$ for $\mu$-a.e. $x\in X$. From the ergodicity of $\mu$, there exists a constant $a_k\geq 0$ such that $F(x)\equiv a_k$ for $\mu$-a.e. $x\in X$.

By Lemma \ref{sc2.1}, there exists a partition  $\beta\in \mathcal{P}^*(\bigvee_{i=0}^{n-1}T^i\mathcal{U}_k)$ such that for $\mu$-a.e. $x\in X$,
\begin{align*}
&\log P_n(T^{-1},f,\mathcal{U}_k,A(x))\\
&=\log \sum_{B\in \beta}\sup_{y\in B\cap A(x)}\exp \sum_{i=0}^{n-1}f^{-i}(y)\\
&\geq \sum_{B\in \beta}\mu_x(B)(\sup_{y\in B\cap A(x)}\exp \sum_{i=0}^{n-1}f^{-i}(y)-\log \mu_x(B))\\
&\quad\quad \quad\quad\quad\quad \quad\quad\quad\quad\quad\quad \quad\quad\,\,(\text{by \eqref{eq4.2} and Lemma \ref{4.1} })\\
&= H_{\mu_x}(\beta)+\sum_{B\in \beta}\sup_{y\in B\cap A(x)}\exp \sum_{i=0}^{n-1}f^{-i}(y)\cdot \mu_x(B)\\
&\geq H_{\mu_x}(\bigvee_{i=0}^{n-1}T^i\mathcal{U}_k)+\int_X f_n\circ T^{-(n-1)}d\mu_X
\end{align*}
Then
\begin{align*}
&a_k=\int_XF(x)d\mu\\
&=\int_X\limsup_{n\rightarrow +\infty} F_n(x)d\mu(x)
\geq \limsup_{n\rightarrow +\infty} \int_XF_n(x)d\mu(x)\\
&\geq \limsup_{n\rightarrow +\infty} \frac{1}{n} \int_X\big(H_{\mu_x}(\bigvee_{i=0}^{n-1}T^i\mathcal{U}_k) +\int_X f_n\circ T^{-(n-1)}d\mu_x\big)d\mu(x)\\
&=\limsup_{n\rightarrow +\infty} \frac{1}{n}\big( \int_X H_{\mu_x} (\bigvee_{i=0}^{n-1}T^i\mathcal{U}_k)d\mu(x) +\int_X  f_n\circ T^{-(n-1)} d\mu(x) \big)\\
&=\limsup_{n\rightarrow +\infty} \frac{1}{n}\int_X H_{\mu_x} (\bigvee_{i=0}^{n-1}T^i\mathcal{U}_k)d\mu(x)  +   \int_X f d\mu(x) \,\, (\text{since }\mu \in \mathcal{M}(X,T))\\
&=\limsup_{n\rightarrow +\infty} \frac{1}{n} H_{\mu}(\bigvee_{i=0}^{n-1}T^i\mathcal{U}_k\mid \mathcal{P}^-)+ \int_X f d\mu(x)  \,\, (\text{by Lemma } \ref{lem2.2}\,(a)   )\\
&=\limsup_{n\rightarrow +\infty} \frac{1}{n} H_{\mu}(\bigvee_{i=0}^{n-1}T^i\mathcal{U}_k\mid T^{-(n-1)}\mathcal{P}^-)+ \int_X f d\mu(x)\\
&\geq H_{\mu}(P_k\mid \mathcal{P}^-)-\frac{1}{k} + \int_X f d\mu(x) \,\,(\text{by \eqref{eq4.3}}).
\end{align*}
Therefore, for $\mu$-a.e. $x\in X$,
\begin{equation*}
P(T^{-1},f,A(x))\geq P(T^{-1},f,\mathcal{U}_k,A(x))=F(x)\geq H_{\mu}(P_k\mid \mathcal{P}^-)-\frac{1}{k} + \int_X f d\mu(x)
\end{equation*}
for each $k\in \mathbb{N}$.

Then
\begin{align*}
P(T^{-1},f,A(x))&\geq \lim_{n\rightarrow +\infty}( H_{\mu}(P_k\mid \mathcal{P}^-)-\frac{1}{k}) + \int_X f d\mu(x)\\
&=H_{\mu}(T)+\int_X f d\mu(x)=P_{\mu}(T,f).
\end{align*}
We complete the proof of Theorem \ref{th4.2}.
\end{proof}

A direct consequence of Theorem \ref{th4.2} is the following.

\begin{corollary}
Let $(X,T)$ be a TDS, $f\in C(X,\mathbb{R})$. If there exists $\mu \in \mathcal{M}^e(X,T)$ with $P_{\mu}(T,f)=P(T,f)$, then there exists $x\in X$, a closed subset $A(x)\subseteq W^s(x,T)$ and a closed subset $B(x)\subseteq W^u(x,T)$ such that
\begin{enumerate}
\item $\lim\limits_{n\rightarrow +\infty} \text{diam}(T^n A(x))=0$ and $P(T^{-1},f,A(x))=P(T,f)$;
\item $\lim\limits_{n\rightarrow +\infty} \text{diam}(T^{-n} B(x))=0$ and $P(T,f,B(x))=P(T,f)$.
\end{enumerate}
\end{corollary}

A TDS $(X,T)$ is transitive if for each  pair of non-empty open subsets $U$ and $V$ of $X$, there exists $n\geq 0$ such that $U\cap T^{-n}V\neq \emptyset$; and it is weakly mixing, if $(X\times X, T\times T)$ is transitive. These notions describe the global properties of the whole TDS. Blanchard \emph{et al.} \cite{Blanchard2008} give a new criterion to picture `a certain amount of weakly mixing' in some consistent sense. The notion of the weakly mixing set was introduced as follows.

If $X$, $Y$ are topological spaces, denote by $\mathcal{C}(X,Y)$ the set of all continuous maps from $X$ to $Y$.
\begin{definition}
Let $(X,T)$ be a TDS and $A\in 2^X$. The set $A$  is said to be weakly mixing  for $T$ if there exists $B\subset A$ satisfying
\begin{enumerate}
\item $B$ is the union of countably many Cantor sets;
\item the closure of $B$ equals $A$;
\item for any $C\in B$ and $g\in \mathcal{C}(C,A)$ there exists an increasing sequence of natural numbers $\{n_i\}\subset \mathbb{N}$ such that $\lim_{i\rightarrow +\infty}T^{n_i}x=g(x)$ for any $x\in C$.
\end{enumerate}
\end{definition}

Denote by $WM_s(X,T)$ the family of weakly mixing subsets of $(X,T)$. The system $(X,T)$ itself is called partially mixing if when it contains a weakly mixing set. The whole space $X$ is a weakly mixing set if and only if TDS $(X,T)$ is weakly mixing \cite{Xiong}. The following result (See Prop. 4.2 \cite{Blanchard2008}) give an equivalent characterization of the weakly mixing set in another way.

\begin{proposition}
Let $(X,T)$ be a TDS and $A$ be a non-singleton closed subset of $X$. Then $A$ is a weakly mixing subset of $X$ if and only if for any $k\in \mathbb{N}$, any choice of non-empty open subsets $V_1,\cdots, V_k$ of $A$ and non-empty open subsets $U_1,\cdots, U_k$ of $X$ with $A\cap U_i\neq \emptyset$, $i=1,2,\cdots,k$, there exists $m\in \mathbb{N}$ such that $T^mV_i\cap U_i\neq \emptyset$ for each $1\leq i\leq k$.
\end{proposition}

Now we prove the following theorem. The first part \ref{4.6a} of Theorem \ref{4.6} was already proved in \cite{Huang2008CMP}. For completeness, we state it in the theorem.

\begin{theorem}\label{4.6}
Let $(X,T)$ be a TDS and $\mu \in \mathcal{M}^e(X,T)$ with
$h_{\mu}(T)>0$. Then for $\mu$-a.e. $x\in X$, there exists a closed
subset $E(x)\subseteq \overline{W^s(x,T)}\cap \overline{W^u(x,T)}$
such that
\begin{enumerate}
\item
$E(x)\in WM_s(X,T)\cap WM_s(X,T^{-1})$, i.e., $E(x)$ is weakly
mixing for $T$, $T^{-1}$.\label{4.6a}
\item $P(T,f,E(x))\geq P_{\mu}(T,f)$ and $P(T^{-1},f,E(x))\geq P_{\mu}(T,f)$.
\end{enumerate}
\end{theorem}

\begin{proof}
Let $\mathcal{B}_{\mu}$ be the completion of $\mathcal{B}_X$ under $\mu$. Then $(X,\mathcal{B}_{\mu},\mu,T)$ is a Lebesgue system. Let $P_{\mu}(T)$ be the Pinsker $\sigma$-algebra of $(X,\mathcal{B}_{\mu},\mu,T)$. Let $\mu=\int_X\mu_xd\mu(x)$ be the disintegration $\mu$ over $P_{\mu}(T)$. Then for $\mu$-a.e. $x\in X$, supp$(\mu_x)\subseteq \overline{W^s(x,T)}\cap \overline{W^u(x,T)}$ and supp$(\mu_x)\in WM_s(X,T)\cap WM_s(X,T^{-1})$(See Theorem 4.6 in \cite{ Huang2008CMP} for details).

We now prove that for $\mu$-a.e. $x\in X$, $P(T,f,\text{supp}(\mu_x))\geq P_{\mu}(T,f)$
and $P(T^{-1},f, \\
\text{supp}(\mu_x))\geq P_{\mu}(T,f)$. By the symmetry of $T$ and $T^{-1}$, $P_{\mu}(T,f)=P_{\mu}(T^{-1},f)$. It remains to prove that for $\mu$-a.e. $x\in X$, $P(T,f,\text{supp}(\mu_x))\geq P_{\mu}(T,f)$. Since $P_{\mu}(T)$ is $T$-invariant,  $T\mu_x=\mu_{Tx}$ for $\mu$-a.e. $x\in X$, therefore, there exists a $T$-invariant measurable set $X_0\subset X$ with $\mu(X_0)=1$ and $T\mu_x=\mu_{Tx}$ for $x\in X_0$.

For each $\mathcal{U}\in \mathcal{C}^o_X$, $x\in X_0$ and $n\in \mathbb{N}$. By Lemma \ref{sc2.1}, there exists a $\beta \in \mathcal{P}^*(\mathcal{U}^{n-1}_0)$ such that
\begin{align}\label{eq4.1}
&\log P_n(T,f,\mathcal{U},\text{supp}(\mu_x))   \notag\\
&=\log \inf\{\sum_{V\in \mathcal{V}}\sup_{y\in V\cap\text{supp}(\mu_x)}\exp f_n(x): \mathcal{V}\in \mathcal{C}_X \,\, \text{and}\,\, \mathcal{V}\succeq \mathcal{U}^{n-1}_0\}         \notag\\
&=\log \sum_{B\in \beta}\sup_{y\in B\cap \text{supp}(\mu_x)}\exp f_n(x)          \notag\\
&\geq \sum_{B\in \beta} \mu_x(B)\big(\sup_{y\in B\cap \text{supp}(\mu_x)}f_n(x)-\log \mu_x(B) \big)
\quad(\text{by Lemma} \,\,\ref{4.1})           \notag\\
&=H_{\mu_x}(\beta)+\sum_{B\in \beta}\mu_x(B)\cdot\sup_{y\in B\cap \text{supp}(\mu_x)}f_n(x) \notag\\
&\geq H_{\mu_x}(\mathcal{U}^{n-1}_0)+\int_Xf_nd\mu_x
\end{align}

Fix $\mathcal{U}\in \mathcal{C}^o_X$ and $n\in \mathbb{N}$, denote $F_n(x)=H_{\mu_x}(\bigvee\limits_{i=0}^{n-1}T^{-i}\mathcal{U})$ + $\int_Xf_nd\mu_x$ for each $x\in X_0$. Then
\begin{align*}
&F_{n+m}(x)=H_{\mu_x}(\bigvee_{i=0}^{n+m-1}T^{-i}\mathcal{U})+\int_Xf_{n+m}d\mu_x\\
&\leq H_{\mu_x}(\bigvee_{i=0}^{n-1}T^{-i}\mathcal{U})+H_{\mu_x}(T^{-n}\bigvee_{i=0}^{m-1}T^{-i}\mathcal{U}) + \int_Xf_nd\mu_x  +\int_X f_m\circ T^n d\mu_x\\
&\leq F_n(x)+H_{T^n\mu_x}(\bigvee_{i=0}^{m-1}T^{-i}\mathcal{U}) + \int_Xf_m\circ T^n d\mu_x\\
&=F_n(x)+H_{T^n\mu_x}(\bigvee_{i=0}^{m-1}T^{-i}\mathcal{U}) + \int_Xf_m dT^n\mu_x\\
&=F_n(x)+H_{\mu_{T^nx}}(\bigvee_{i=0}^{m-1}T^{-i}\mathcal{U}) + \int_Xf_m d\mu_{T^nx}\\
&=F_n(x)+F_m(T^nx)
\end{align*}
That is, $\{F_n\}$ is subadditive. Since the map $x\rightarrow\mu_x(A)$ for each $A\in \mathcal{B}$ is measurable on $X_0$, it follows that $F_n(x)$ is measurable on $X_0$. By Kingman sub-additive ergodic theorem, $\lim\limits_{n\rightarrow \infty}\frac{1}{n}F_n(x)\equiv a_{\mathcal{U}}$ for $\mu$-a.e. $x\in X$, where $a_{\mathcal{U}}$ is a constant. Then by \eqref{eq4.1}, $P(T,f,\mathcal{U},\text{supp}(\mu_x))\geq a_{\mathcal{U}}$ for each $\mathcal{U}\in \mathcal{C}^0_X$ and $\mu$-a.e. $x\in X$. Therefore
\begin{align*}
a_{\mathcal{U}}&=\int_X\lim_{n\rightarrow \infty}\frac{1}{n}F_n(x) d\mu\\
&=\lim_{n\rightarrow \infty}\frac{1}{n}\int_X F_n(x) d\mu \quad \text{(by Kingman sub-additive ergodic theorem)}\\
&=\lim_{n\rightarrow \infty}\frac{1}{n}\int_X\big(H_{\mu_x}(\mathcal{U}^{n-1}_0) +\int_X f_n d\mu_x \big)d\mu(x)\\
&=\lim_{n\rightarrow \infty}\frac{1}{n}H_{\mu}(\mathcal{U}_0^{n-1}\mid P_{\mu}(T)) + \int_X f d\mu\\
&=h_{\mu}(T,\mathcal{U}\mid P_{\mu}(T)) + \int_X f d\mu\\
&=P_{\mu}(T,f,\mathcal{U}) \quad (\text{by Lemma \ref{4.2}}).
\end{align*}
It follows that $P(T,f,\mathcal{U},\text{supp}(\mu_x))\geq P_{\mu}(T,f,\mathcal{U})$ for each $\mathcal{U}\in \mathcal{C}_X^o$ and $\mu$-a.e. $x\in X$.

Choose a sequence of open cover $\{\mathcal{U}_m\}_{m=1}^{\infty}$ with $\lim\limits \text{diam}\{\mathcal{U}_m\}=0$. Then
\begin{align*}
\lim_{n\rightarrow\infty} P_{\mu}(T,f,\mathcal{U}_m)&=\lim_{n\rightarrow\infty} (h_{\mu}(T,\mathcal{U}_m)+\int_X f d\mu)\\
&=h_{\mu}(T)+\int_Xfd\mu =P_{\mu}(T,f).
\end{align*}
Since for each $m\in \mathbb{N}$ and $\mu$-a.e. $x\in X$, $P(T,f,\mathcal{U}_m,\text{supp}(\mu_x))\geq P_{\mu}(T,f,\mathcal{U}_m)$, we have
\begin{equation*}
P(T,f,\text{supp}(\mu_x))=\sup_{m\in \mathbb{N}}P(T,f,\mathcal{U}_m,\text{supp}(\mu_x))\geq \sup_{m\geq 1}P_{\mu}(T,f,\mathcal{U}_m)=P_{\mu}(T,f)
\end{equation*}
for each $\mu$-a.e. $x\in X$.
\end{proof}

It is not hard to see that the following corollary holds.

\begin{corollary}
Let $(X,T)$ be a TDS and $f\in C(X,\mathbb{R})$. Then
\begin{enumerate}
\item $\sup_{ x\in X}P(T,f,\overline{W^s(x,T)}\cap \overline{W^u(x,T)})=P(T,f)$;
\item If there exists $\mu\in \mathcal{M}^e(X,T)$ with $P_{\mu}(T,f)=P(T,f)$,
then for $\mu$-a.e. $x\in X$, there exists a closed subsets
$E(x)\subseteq \overline{W^s(x,T)}\cap \overline{W^u(x,T)}$ such
that
\begin{enumerate}
\item
$E(x)\in WM_s(X,T)\cap WM_s(X,T^{-1})$ and
\item $P(T,f,E(x))=P(T^{-1},f,E(x))=P(T,f)$.
\end{enumerate}
\end{enumerate}
\end{corollary}

\section{Non-invertible case}\label{sec5}

In this section, we will generalize the results in Sects. \ref{sec3} and \ref{sec4} to the non-invertible case. Let $(X,T)$ be a non-invertible TDS, i.e., $X$ is a compact metric space, and $T:X\rightarrow X$ is a surjective continuous map but not one-to-one.

Let $d$ be the metric on $X$ and define
$\widetilde{X}=\{(x_1,x_2,\cdots): T(x_{i+1})=x_i, x_i\in X, i\in
\mathbb{N}\}$. It is clear that $\widetilde{X}$ is a subspace of the
product space $\Pi_{i=1}^{\infty}X$ with the metric $d_T$ defined by
$$
d_T((x_1,x_2,\cdots),(y_1,y_2,\cdots))=\sum_{i=1}^{\infty}\frac{d(x_i,y_i)}{2^i}.
$$
Let $\widetilde{T}: \widetilde{X}\rightarrow\widetilde{X}$ be the shift
homeomorphism, i.e., $\widetilde{T}(x_1,x_2,\cdots)=
(T(x_1),x_1,x_2,\cdots).$ We refer the TDS
$(\widetilde{X},\widetilde{T})$ as the  inverse limit of
$(X,T)$. Let $\pi_i:\widetilde{X}\rightarrow X$ be the natural
projection onto the $i^{\text{th}}$ coordinate. Then
$\pi_i:(\widetilde{X},\widetilde{T})\rightarrow (X,T)$ is a factor map.

We need the following lemmas.
\begin{lemma}\label{5.02}
Let $(X,T)$ be a non-invertible TDS, $f\in C(X,\mathbb{N})$. Then
for each $\mathcal{U}\in \mathcal{C}_X^o$ and each compact subset
$K$ of $X$,
\begin{equation*}
P_{n+m}(T,f,\mathcal{U},K)\leq P_m(T,f,\mathcal{U},K)\cdot
P_n(T,f\circ T^m, T^{-m}\mathcal{U},K)
\end{equation*}
for each $ n,m\in \mathbb{N}$.
\end{lemma}

\begin{proof}
Since for each $\mathcal{V}_1\succeq \mathcal{U}_0^{m-1}$,
$\mathcal{V}_2\succeq \mathcal{U}_0^{n-1}$, then $\mathcal{V}_1\vee
T^{-m}\mathcal{V}_2\succeq \mathcal{U}_0^{n+m-1}$. It follows that
\begin{equation*}
\begin{split}
P_{n+m}(T,f,\mathcal{U},K)&\leq \sum_{V_1\in
\mathcal{V}_1}\sum_{V_2\in \mathcal{V}_2}\sup_{x\in V_1\cap
T^{-m}V_2\cap K}\exp f_{n+m}(x)\\
&=\sum_{V_1\in \mathcal{V}_1}\sum_{V_2\in \mathcal{V}_2}\sup_{x\in
V_1\cap T^{-m}V_2\cap K}\exp (f_m(x)+f_n(T^mx))\\
&\leq \sum_{V_1\in \mathcal{V}_1}\sum_{V_2\in
\mathcal{V}_2}\sup_{x\in V_1\cap K} \exp f_m(x)\cdot \sup_{x\in
T^{-m}V_2\cap K} \exp f_n(T^mx)\\
&=\sum_{V_1\in \mathcal{V}_1}\sup_{x\in V_1\cap K} \exp
f_m(x)\cdot\sum_{V_2\in \mathcal{V}_2}\sup_{x\in T^{-m}V_2\cap K}
\exp (f\circ T^m)_n(x).
\end{split}
\end{equation*}
By the arbitrary of $\mathcal{V}_1$ and $\mathcal{V}_2$, we have
\begin{equation*}
P_{n+m}(T,f,\mathcal{U},K)\leq P_m(T,f,\mathcal{U},K)\cdot
P_n(T,f\circ T^m, T^{-m}\mathcal{U},K).
\end{equation*}
\end{proof}

\begin{lemma}\label{5.01}
Let $(X,T)$ be a non-invertible TDS, $f\in C(X,\mathbb{N})$. Then
for each $\mathcal{U}\in \mathcal{C}_X^o$ and each compact subset
$K$ of $X$,
\begin{equation*}
P_n(T,f\circ
T^m,T^{-m}\mathcal{U},T^{-m}K)=P_n(T,f,\mathcal{U},K)\,\,\,
\text{for each}\,\, n,m\in \mathbb{N}.
\end{equation*}
\end{lemma}

\begin{proof}
Fix $n,m\in \mathbb{N}$. For each $\mathcal{V}\succeq
(T^{-m}\mathcal{U})_0^{n-1}$,
\begin{equation*}
\begin{split}
&\sum_{V\in\mathcal{V}}\sup_{x\in V\cap T^{-m}K}\exp(f\circ
T^m)_n(x)\\
=&\sum_{V\in\mathcal{V}}\sup_{x\in V\cap T^{-m}K}\exp f_n(T^m
x)\\
=&\sum_{V\in\mathcal{V}}\sup_{ x\in T^m V \cap K}\exp f_n(x).
\end{split}
\end{equation*}
Since $T^m\mathcal{V}\succeq \mathcal{U}_0^{n-1}$, then
\begin{equation*}
P_n(T,f\circ T^m,T^{-m}\mathcal{U},T^{-m}K)\leq
P_n(T,f,\mathcal{U},K).
\end{equation*}

Conversely, for each $\mathcal{V}\succeq \mathcal{U}_0^{n-1}$,
$T^{-m}\mathcal{V}\succeq (T^{-m}\mathcal{U})_0^{n-1}$ and
\begin{equation*}
\begin{split}
&\sum_{V\in\mathcal{V}}\sup_{x\in V\cap  K}\exp f_n(x)\\
=&\sum_{V\in\mathcal{V}}\sup_{x\in T^{-m}(V\cap K)}\exp f_n(T^m
x)\\
=&\sum_{V\in\mathcal{V}}\sup_{ x\in T^{-m} V \cap T^{-m}K}\exp
(f\circ T^m)_n(x).
\end{split}
\end{equation*}
Then
\begin{equation*}
P_n(T,f\circ T^m,T^{-m}\mathcal{U},T^{-m}K)\geq
P_n(T,f,\mathcal{U},K),
\end{equation*}
and we complete the proof.
\end{proof}

\begin{lemma}\label{5.1}
Let $(\widetilde{X},\widetilde{T})$ be the inverse limit of a
non-invertible TDS $(X,T)$, $f\in C(X,\mathbb{R})$ and $\pi_1:
\widetilde{X}\rightarrow X$ be the projection to the
1\textsuperscript{th} coordinate. Then for any sequence non-empty
closed subsets $K_n$ of $\widetilde{X}$,
\begin{equation*}
\lim_{\delta\rightarrow 0}\limsup_{n\rightarrow
+\infty}\frac{1}{n}\log P_n(\widetilde{T},f\circ
\pi_1,\delta,K_n)=\lim_{\delta\rightarrow 0}\limsup_{n\rightarrow
+\infty}\frac{1}{n}\log P_n(T,f,\delta,\pi_1(K_n)).
\end{equation*}
\end{lemma}

\begin{proof}
Let $\mathcal{U}\in \mathcal{C}_X^o$. For each $\mathcal{V}\in
\mathcal{C}_X$ with $\mathcal{V}\succeq \mathcal{U}^{n-1}_0$  and
$x\in V\cap \pi_1(K_n)$, obviously,$\pi_1^{-1}\mathcal{V}\succeq
(\pi_i^{-1}\mathcal{U})_0^{n-1}$ and
\begin{equation*}
(f\circ\pi_1)_n(\widetilde{x})=\sum_{j=0}^{n-1}(f\circ\pi_1)(\widetilde{T}^j(\widetilde{x}))
=\sum_{j=0}^{n-1}f\circ
T^j(\pi_1\widetilde{x})=f_n(\pi_1\widetilde{x})=f_n(x),
\end{equation*}
where $x=\pi_1\widetilde{x}$. Then
\begin{equation*}
\sum_{V\in \mathcal{V}}\sup_{\widetilde{x}\in \pi_1^{-1}V\cap
K_n}\exp (f\circ \pi_1)_n(\widetilde{x})=\sum_{V\in
\mathcal{V}}\sup_{x\in V\cap\pi_1(K_n)}\exp f_n(x).
\end{equation*}
It follows that
\begin{equation}\label{eq5.3}
P_n(\widetilde{T},f\circ \pi_1,\pi_1^{-1}\mathcal{U},K_n)\leq
P_n(T,f,\mathcal{U},\pi_1(K_n)).
\end{equation}
On the other hand, for each $\widetilde{\mathcal{V}}\in
\mathcal{C}_{X}^{\mathbb{N}}$ with $\widetilde{\mathcal{V}}\succeq
(\pi_1^{-1}\mathcal{U})_0^{n-1}$ and
$\widetilde{x}\in\widetilde{V}\cap K_n$, $\pi_1\widetilde{V}\succeq
\mathcal{U}_0^{n-1}$ and
\begin{equation*}
\begin{split}
\sum_{\widetilde{V}\in
\widetilde{\mathcal{V}}}\sup_{\widetilde{x}\in\widetilde{V}\cap
K_n}\exp (f\circ\pi_1)_n(\widetilde{x})=&\sum_{\widetilde{V}\in
\widetilde{\mathcal{V}}}\sup_{x\in\pi_1(\widetilde{V}\cap K_n)}\exp
f_n(x)\\
=&\sum_{V\in\pi_1\widetilde{V}}\sup_{x\in\pi_1\widetilde{V}\cap
\pi_1K_n}\exp f_n(x),
\end{split}
\end{equation*}
where $x=\pi_1\widetilde{x}$. Then we get the opposite part of the
inequality of \eqref{eq5.3}, and consequently
\begin{equation}\label{eq5.4}
P_n(\widetilde{T},f\circ \pi_1,\pi_1^{-1}\mathcal{U},K_n)=
P_n(T,f,\mathcal{U},\pi_i(K_n)).
\end{equation}

Now we have
\begin{equation*}
\limsup_{n\rightarrow \infty}\frac{1}{n}\log
P_n(\widetilde{T},f\circ
\pi_1,\pi_1^{-1}\mathcal{U},K_n)=\limsup_{n\rightarrow
\infty}\frac{1}{n}\log P_n(T,f,\mathcal{U},\pi_1(K_n)).
\end{equation*}

From Lemma \ref{lem3.1}, we get
\begin{equation*}
\lim_{\delta\rightarrow 0}\limsup_{n\rightarrow
\infty}\frac{1}{n}\log P_n(\widetilde{T},f\circ
\pi_1,\delta,K_n)\geq\lim_{\delta\rightarrow 0}\limsup_{n\rightarrow
\infty}\frac{1}{n}\log P_n(T,f,\delta,\pi_1(K_n)).
\end{equation*}

Conversely, let $\pi_i:\widetilde{X}\rightarrow X$ be the projection
to the i\textsuperscript{th} coordinate and
$\widetilde{\mathcal{U}}\in \mathcal{C}_{\widetilde{X}}^o$. By the
definition of $\widetilde{X}$, it is easy to see that there exists
some $\mathcal{U}\in \mathcal{C}_X^o$ such that
$\pi_i^{-1}(\mathcal{U})\succeq \widetilde{U}$. Since for any two
closed subsets $C$ and $D$ of $X$, $P_n(T,f,\mathcal{U},C)\leq
P_n(T,f,\mathcal{U},D)$, and $\pi_i(K_n)\succeq
T^{-(i-1)}\pi_i(K_n)$, then by \eqref{eq5.4}, we have
\begin{equation*}
\begin{split}
&\limsup_{n\rightarrow \infty}\frac{1}{n}\log
P_n(\widetilde{T},f\circ\pi_1,\widetilde{\mathcal{U}},K_n)\\
\leq &\limsup_{n\rightarrow \infty}\frac{1}{n}\log
P_n(\widetilde{T},f\circ\pi_1,\pi_i^{-1}\mathcal{U},K_n)\\
=&\limsup_{n\rightarrow \infty}\frac{1}{n}\log
P_n(T,f,\mathcal{U},\pi_i(K_n))\\
\leq &\limsup_{n\rightarrow \infty}\frac{1}{n}\log
P_n(T,f,\mathcal{U},T^{-(i-1)}\pi_i(K_n))\\
=& \limsup_{n\rightarrow \infty}\frac{1}{n+i-1}\log
P_{n+i-1}(T,f,\mathcal{U},T^{-(i-1)}\pi_i(K_n) )\\
\leq & \limsup_{n\rightarrow \infty}\frac{1}{n+i-1}\log \big(
P_{i-1}(T,f,\mathcal{U},T^{-(i-1)}\pi_i(K_n))\\
&\quad \quad\quad \quad \quad\cdot P_n(T,f\circ
T^{i-1},T^{-(i-1)}\mathcal{U},T^{-(i-1)}\pi_i(K_n))
 \big)\quad\text{(by Lemma \ref{5.02})}\\
=&\limsup_{n\rightarrow \infty} \frac{1}{n}
P_n(T,f,\mathcal{U},\pi_1(K_n)) \quad  \quad\text{(by Lemma
\ref{5.01})}\\
\leq & \lim_{\delta\rightarrow
0}\limsup_{n\rightarrow\infty}\frac{1}{n}P_n(T,f,\delta,\pi_1(K_n)
).
\end{split}
\end{equation*}

By Lemma \ref{lem3.1}, we get
\begin{equation*}
\lim_{\delta\rightarrow 0}\limsup_{n\rightarrow
\infty}\frac{1}{n}\log P_n(\widetilde{T},f\circ
\pi_1,\delta,K_n)\leq\lim_{\delta\rightarrow 0}\limsup_{n\rightarrow
\infty}\frac{1}{n}\log P_n(T,f,\delta,\pi_1(K_n)).
\end{equation*}
\end{proof}

Now we can prove the following theorem.

\begin{theorem}\label{5.2}
Let $(X,T)$ be a non-invertible TDS, $f\in C(X,\mathbb{R})$ and $\mu
\in \mathcal{M}^e(X,T)$ with $h_{\mu}(T)>0$. Then for $\mu$-a.e.
$x\in X$, $\lim_{\epsilon\rightarrow 0}P_s(T,f,x,\epsilon)\geq
P_{\mu}(T,f)$.
\end{theorem}

\begin{proof}
Let $(\widetilde{X},\widetilde{T})$ be the inverse limit of $(X,T)$.
For $\epsilon>0$, $n\in \mathbb{N}$ and $\widetilde{x}\in
\widetilde{X}$,denote
$K_n=\widetilde{T}^{-n}W_{\frac{\epsilon}{2}}^s(\widetilde{x},\widetilde{T})$.
Then from the definition of $d_T$ and $\widetilde{X}$, it is easy to
see that $\pi_1(K_n)\subseteq T^{-n}W_{\epsilon}^s(x,T)$, where
$x=\pi_1(\widetilde{x})$. By Lemma \ref{5.1}, we have
\begin{align*}
P_s(T,f,x,\epsilon)=&\lim_{\delta\rightarrow 0}\limsup_{n\rightarrow
+\infty}\frac{1}{n}\log P_n(T,f,\delta,T^{-n}W_{\epsilon}^s(x,T))\\
\geq & \lim_{\delta\rightarrow 0}\limsup_{n\rightarrow +\infty}\log
P_n(T,f,\delta,\pi_1(K_n))\\
=& \lim_{\delta\rightarrow 0}\limsup_{n\rightarrow +\infty}\log
P_n(\widetilde{T},f\circ \pi_1,\delta, K_n)\\
=&
P_s(\widetilde{T},f\circ \pi_1,\widetilde{x},\frac{\epsilon}{2}).
\end{align*}
It follows that for each $\widetilde{x}\in \widetilde{X}$,
\begin{equation}\label{eq5.1}
\lim_{\epsilon \rightarrow
0}P_s(T,f,\pi_1(\widetilde{x}),\epsilon)\geq
\lim_{\epsilon\rightarrow
0}P_s(\widetilde{T},f\circ\pi_1,\widetilde{x},\frac{\epsilon}{2}).
\end{equation}
Let $\widetilde{\mu}\in\mathcal{M}^e(\widetilde{X},\widetilde{T})$
with $\pi_1(\widetilde{\mu})=\mu$. Then by Theorem \ref{th3.7}, there
exists a Borel subset $\widetilde{X_0}\subseteq \widetilde{X}$ with
$\widetilde{\mu}(\widetilde{X_0})=1$ such that for any
$\widetilde{x}\in \widetilde{X_0}$,
\begin{equation}\label{eq5.2}
\begin{split}
\lim_{\epsilon \rightarrow
0}P_s(\widetilde{T},f\circ\pi_1,\widetilde{x},\frac{\epsilon}{2})\geq
P_{\widetilde{\mu}}(\widetilde{T},&f\circ\pi_1)
=h_{\widetilde{\mu}}(\widetilde{T})+\int_{\widetilde{X}}f\circ\pi_1
d\widetilde{\mu}\\
 \geq & h_{\mu}(T)+\int_Xfd\mu
 =P_{\mu}(T,f).
\end{split}
\end{equation}
Let $X_0=\pi_1(\widetilde{X_0})$. Then $X_0\in \mathcal{B}_{\mu}$
and $\mu(X_0)=1$. By the inequality \eqref{eq5.1} and \eqref{eq5.2},
we have
\begin{equation*}
\lim_{\epsilon\rightarrow 0}P_s(T,f,x,\epsilon)\geq P_{\mu}(T,f)
\quad \text{for each} \,\, x\in X_0,
\end{equation*}
and we complete the proof.
\end{proof}

Theorem \ref{5.2} immediately lead to the following corollary.

\begin{corollary}\label{5.3}
Let $(X,T)$ be a non-invertible TDS and $f\in C(X,\mathbb{R})$. If
there exists a $\mu \in \mathcal{M}^e(X,T)$ such that
$P_{\mu}(T,f)=P(T,f)$, then $\mathcal{P}(T,f)\neq \emptyset$.
\end{corollary}

\begin{lemma}\label{5.4}
Let $(\widetilde{X},\widetilde{T})$ be the inverse limit of a
non-invertible TDS $(X,T)$. If $A\subseteq\widetilde{E}$ is weak
mixing, so is $\pi_1(A)$ and $P(\widetilde{T},f\circ
\pi_1,A)=P(T,f,\pi_1(A))$.
\end{lemma}

\begin{proof}
The fact that $\pi_1(A)$ is weak mixing follows from Lemma 4.8 in
\cite{Blanchard2008}. The latter follows from Lemma \ref{5.1} and Lemma \ref{lem3.1}.
\end{proof}

The following theorem shows that Theorem \ref{4.6} also holds for non-invertible TDS.

\begin{theorem}
Let $(X,T)$ be a non-invertible TDS and $\mu\in \mathcal{M}^e(X,T)$
with $h_{\mu}(T)>0$. Then for $\mu$-a.e. $x\in X$, there exists a
closed subset $E(x)\subseteq \overline{W^s(x,T)}$ such that
$P(T,f,E(x))\geq P_{\mu}(T,f)$ and $E(x)\in WM_s(X,T)$.
\end{theorem}

\begin{proof}
Let $(\widetilde{X},\widetilde{T})$ be the inverse limit of $(X,T)$.
Then there exists $\widetilde{\mu}\in
\mathcal{M}^e(\widetilde{X},\widetilde{T})$ with
$\pi_1(\widetilde{\mu})=\mu$, where $\pi_1$ is the projection to the
1\textsuperscript{th} coordinate. Obviously,
\begin{equation*}
P_{\widetilde{\mu}}(\widetilde{T},f\circ\pi_1)
=h_{\widetilde{\mu}}(\widetilde{T})+\int_{\widetilde{X}}f\circ\pi_1d\widetilde{\mu}
\geq h_{\mu}(T)+\int_X f d\mu =P(T,f).
\end{equation*}
By Theorem 4.6, there exists a Borel set $\widetilde{X_0}\subseteq
\widetilde{X}$ with $\widetilde{\mu}(\widetilde{X_0})=1$ such that
for each $\widetilde{x}\in \widetilde{X_0}$, there exists  a closed
subset $E(\widetilde{x})\subseteq
\overline{W^s(\widetilde{x},\widetilde{T})}$ such that
\begin{equation*}
P(\widetilde{T},f\circ\pi_1, E(\widetilde{x}))\geq
P_{\widetilde{\mu}}(\widetilde{T},f\circ\pi_1) \quad\text{and}\,\,
E(\widetilde{x})\in WM_s(\widetilde{X},\widetilde{T}).
\end{equation*}

Let $(X_0)=\pi_1(\widetilde{X_0})$. Then $X_0\in \mathcal{B}_{\mu}$
and $\mu(X_0)=1$. For each $x\in X_0$ let
$E(x)=\pi_1(E(\widetilde{x}))$, where $x=\pi_1(\widetilde{x})$. Then
$E(x)\subseteq
\pi_1(\overline{W^s(\widetilde{x},\widetilde{T})})\subseteq
\overline{W^s(x,T)}$. By Lemma \ref{5.4}, we have
\begin{equation*}
P(T,f,E(x))=P(\widetilde{T},f\circ,E(\widetilde{x}))\geq
P_{\widetilde{\mu}}(\widetilde{T},f\circ \pi_1)\geq P_{\mu}(T,f)
\end{equation*}
and $E(x)\in WM_s(X,T)$.
\end{proof}

The following result is immediate.

\begin{corollary}
Let $(X,T)$ be a non-invertible TDS. Then
\begin{enumerate}
\item $\sup_{ x\in X}P(T,f,\overline{W^s(x,T)})=P(T,f)$;
\item If there exists $\mu\in \mathcal{M}^e(X,T)$ with $P_{\mu}(T,f)=P(T,f)$,
then for $\mu$-a.e. $x\in X$, there exists a closed subsets
$E(x)\subseteq \overline{W^s(x,T)}$ such that $E(x)\in WM_s(X,T)$
and $P(T,f,E(x))=P(T,f)$.
\end{enumerate}

\end{corollary}

\begin{remark}
In fact, from the proof of Theorem \ref{4.6}, we know
$E(x)=supp(\mu_x)$, where $\mu_x$ is a probability measure
determined by the disintegration of $\mu\in \mathcal{M}^e(X,T)$ over
the Pinsker $\sigma$-algebra $P_{\mu}(T)$.

\end{remark}

%

\end{document}